\documentclass[aap]{imsart}

\RequirePackage{amsthm,amsmath,amsfonts,amssymb}
\RequirePackage[numbers]{natbib}

\startlocaldefs
\usepackage[colorlinks=true, linkcolor=blue,citecolor=blue]{hyperref}

\usepackage{amsmath,amssymb,amsfonts,amsthm,mathrsfs,comment,xcolor,color,bm}
\usepackage[shortlabels]{enumitem}
\usepackage{import}

\newcommand{\bbR}{\mathbb{R}}

\newcommand{\bbRD}{\mathbb{R}^d}
\newcommand{\N}{\mathbb{N}}

\newcommand{\cF}{\mathcal{F}}
\newcommand{\cE}{\mathcal{E}}

\newcommand{\cH}{\mathcal{H}}

\newcommand{\cP}{\mathcal{P}}

\newcommand{\cU}{\mathcal{U}}

\newcommand{\cL}{\mathcal{L}}

\newcommand{\bes}{\begin{equation*}}
\newcommand{\ees}{\end{equation*}}
\newcommand{\beas}{\begin{eqnarray*}}
\newcommand{\eeas}{\end{eqnarray*}}
\newcommand{\bea}{\begin{eqnarray}}
\newcommand{\eea}{\end{eqnarray}}
\newcommand{\be}{\begin{equation}}
\newcommand{\ee}{\end{equation}}
\newcommand{\bei}{\begin{itemize}}
\newcommand{\eei}{\end{itemize}}
\newcommand{\bec}{\begin{cases}}
\newcommand{\eec}{\end{cases}}
\newcommand{\ben}{\begin{enumerate}}
\newcommand{\een}{\end{enumerate}}

\newcommand{\bbP}{\mathbb{P}}
\newcommand{\bbE}{\mathbb{E}}

\newcommand{\bbl}{\begin{block}}
\newcommand{\ebl}{\end{block}}

\newcommand{\De}{\mathrm{d}}

\newcommand{\rme}{\mathrm{e}}
\newcommand{\bmX}{\bm{X}}
\newcommand{\bmC}{\bm{C}}

\newcommand{\bmu}{\bm{u}}
\newcommand{\bmf}{\bm{f}}

\newtheorem{mydef}{Definition}[section]
\newtheorem{prop}{Proposition}[section]
\newtheorem{theorem}{Theorem}[section]
\newtheorem{lemma}{Lemma}[section]

\newtheorem{remark}{Remark}[section]
\newtheorem{myxmpl}{Example}[section]
\newtheorem{assumption}[theorem]{Assumption}

\newtheorem{cor}{Corollary}[section]

\newcommand{\ip}[2]{\langle #1,#2\rangle}

\newcommand{\cref}{C_{\mathrm{ref}}}
\newcommand{\op}{\bmX^{0,\xi,T,g}}

\newcommand{\lip}{\mathrm{Lip}}

\endlocaldefs

\begin{document}

\begin{frontmatter}
\title{Coupling by reflection for controlled diffusion processes: turnpike property and large time behavior of Hamilton Jacobi Bellman equations}
\runtitle{Coupling by reflection for controlled diffusion processes}

\begin{aug}
\author[A]{\fnms{Giovanni}~\snm{Conforti}\ead[label=e1]{Giovanni.conforti@polytechnique.edu}},
\address[A]{CMAP, Ecole Polytechnique,IPParis.\printead[presep={,\ }]{e1}}

\end{aug}

\begin{abstract}
  We investigate the long time behavior of weakly dissipative semilinear Hamilton-Jacobi-Bellman (HJB) equations and the turnpike property for the corresponding stochastic control problems. To this aim, we develop a probabilistic approach based on a variant of coupling by reflection adapted to the study of controlled diffusion processes. We prove existence and uniqueness of solutions for the ergodic Hamilton-Jacobi-Bellman equation and different kind of quantitative exponential convergence results at the level of the value function, of the optimal controls and of the optimal processes. Moreover, we provide uniform in time gradient and Hessian estimates for the solutions of the HJB equation that are of independent interest.
\end{abstract}



\end{frontmatter}


%
%

\begin{acks}[Acknowledgments]
Research supported by the ANR project  ANR-20-CE40-0014.
\end{acks}
\section{Introduction and statement of the main results}\label{sec:intro}

Aim of this paper is to analyse the ergodic behavior of a class of stochastic control problems and the associated semilinear PDEs through a probabilistic approach drawing inspiration from recent applications of coupling techniques to quantify the exponential rate of convergence to equilibrium of uncontrolled diffusion processes, see e.g. \cite{eberle2016reflection,eberle2019couplings,eberle2019sticky,guillin2021kinetic}. In particular, we propose novel versions of coupling by reflection and sticky coupling that are well-suited to the study of controlled diffusion processes. 
Understanding the long time averages of dynamic control problems is a classical and yet still rapidly developing research field: for deterministic control problems, there exist at least two different approaches depending on whether the main interest is the ergodic behavior of the value function and the corresponding Hamilton-Jacobi equation or rather the turnpike property for optimal trajectories. The first approach is rooted in the influential works \cite{lions1986homogenization,fathi1998convergence,roquejoffre1998comportement,namah1999remarks,barles2000large}, essentially focusing on the case of periodic Hamiltonians or compact manifolds. On the other hand, the turnpike property refers to the general principle that extremal curves tend to spend most of their time in proximity of equilibrium states, called turnpikes. The first turnpike theorems for problems arising in econometry date back to \cite{dorfman1987linear,mckenzie1963turnpike}. Much more recently, following the detailed analysis of the linear-quadratic setting made in \cite{porretta2013long}, Trélat and Zuazua obtained a general local exponential turnpike estimate in \cite{trelat2015turnpike} for non-linear deterministic control problems. The last few years have witnessed a rapid growth of research activity surrounding this subject in connection with applications ranging from neural network to aircraft design: it would be impossible to account for all relevant contributions here and we rather refer to the  recent surveys \cite{geshkovski2022turnpike,faulwasser2020turnpike}. Ergodicity of stochastic control problems has been mostly investigated through the convergence of the value function to the ergodic solution of the Hamilton-Jacobi-Bellman equation. This is done either by working directly on the action functional to obtain uniform in time estimates on its variations, see \cite{arisawa1998ergodic,barles2001space,fujita2006asymptotic,ichihara2012large,ichihara2013large} or by working on its pathwise representation in terms of backward stochastic differential equations, see 
\cite{fuhrman2009ergodic,debussche2011ergodic,cosso2016long}. Here, the authors work under various dissipativity conditions on the controlled dynamics and, in some cases are able to show that convergence happens exponentially fast \cite{hu2015probabilistic,hu2019ergodic}. When it comes to the ergodicity of optimally controlled processes there are way less quantitative results in the stochastic control literature we are aware of, especially about exponential turnpike estimates. Notable exceptions are the detailed analysis of the linear-quadratic setting carried out in \cite{sun2022turnpike} and the series of papers initiated in \cite{cardaliaguet2012long,cardaliaguet2013long} on mean field games, where the above mentioned dissipativity conditions are replaced by the so called monotonicity conditions on the coupling terms, see also \cite{cardaliaguet2019long} for results on the associated master equation. On another note, the articles \cite{backhoff2020mean,clerc2020long} establish entropic turnpike estimates on the Schr\"odinger problem \cite{LeoSch} and its mean field counterpart, thereby highlighting the relevance of curvature lower bounds for the turnpike phenomenon. Leaving all precise statements to the main body of the article, let us now give a brief overview of our main results.
\bei 
\item Our first contribution is Theorem \ref{thm:SP_val_fun} about the value function. In there, we show well posedness for the ergodic Hamilton Jacobi Bellman equation and exponential convergence in Lipschitz norm. The proof hinges on uniform in time gradient and Hessian bounds for the value function that are of independent interest. 

\item The second main result is Theorem \ref{thm:SP_turnpike_1}, where we establish an exponential turnpike property in Wasserstein distance for optimal controls and optimal processes.
\eei
 Once the full statement of our main results is given in subsections \ref{sub:val_fun} and \ref{sub:turnpike}, we will be in a better position to make a more thorough comparison between our contributions and the above mentioned articles, and we shall do so at the end of this introductory section. The proofs are carried out relying on two different sets of assumptions, whose roles are significantly different. The first is Assumption \ref{ass:SP}, that is roughly what is needed to prove the turnpike property for optimal processes. The second set of Assumptions is \ref{ass:SP_wellposed}, that we need to prove uniform Hessian bounds and simplifies the study of the Hamilton-Jacobi-Bellman equation. However, we stress here that with the sole exception of some multiplicative factors in Theorem \ref{item_4:SP_turnpike}, all constants appearing in the exponential estimates at Theorem and \ref{thm:SP_val_fun} and \ref{thm:SP_turnpike_1} depend only on the requirements of Assumptions \ref{ass:SP} and are completely independent from all constants appearing at Assumption \ref{ass:SP_wellposed}. For this reason, it is natural to postulate that all these exponential convergence results would hold under the sole Assumption \ref{ass:SP}, or just slightly more, by working with viscosity or mild solutions for the HJB equation. However, we do not further elaborate on this point in the present work not to overshadow the main message we want to convey. We conclude this brief introduction pointing out that the methods developed in this article open up new perspectives for a quantitative study of the long time behaviour of stochastic control problems of McKean-Vlasov type. Given the success of techniques based on variants of coupling by reflection in analysing the long time behaviour of the Mc-Kean Vlasov dynamics, see e.g. \cite{durmus2022sticky}, it seems plausible that some progress in this direction could be made: we plan to explore this research line in the near future.
 We now introduce some standard basic notation and eventually proceed to the presentation of the main results.


\paragraph{Notation}
For $d',d,m,n\geq 1 $ we denote by $C^n(\bbRD;\bbR^{d'})$ the space of continuous functions $f:\bbRD\rightarrow\bbR^{d'}$ that are $n$ times differentiable and whose partial derivatives of order $n$ are continuous. For $T>0$, $C^{m,n}([0,T)\times\bbRD;\bbR^{d'})$ is the space of continuous functions $f:[0,T)\times\bbRD\rightarrow\bbR^{d'}$ that are $n$ times differentiable in the space variable and $m$ times differentiable in the time variable with continuous partial derivatives of order $n$ in space and continuous partial derivatives of order $m$ in time. When functions are real valued, we omit to specify this; for example we write $C^n(\bbRD)$ instead of $C^n(\bbRD;\bbR^{d'})$. Moreover, we shall denote by $\mathrm{Lip}(\bbRD)$ the space of Lipschitz real valued functions and by $\|\cdot \|_{\mathrm{Lip}}$ the corresponding Lipschitz norm, i.e. 
\bes
\|g \|_{\mathrm{Lip}}:= \sup_{\substack{x,x'\in\bbRD\\ x\neq x'}} \frac{|g(x)-g(x')|}{|x-x'|}
\ees
 $\mathrm{Lip}(\bbRD;\bbRD)$ shall be used for the set of Lipschitz vector fields and we extend the definition of $\|\cdot \|_{\mathrm{Lip}}$ to this set in the obvious way. $C_p(\bbRD;\bbR^{d'})$ and $C_p([0,T)\times \bbRD;\bbR^{d'})$ are used to real valued denote the set of functions with polynomial growth. Even in this case, we shall use the shorthand notations $C_p([0,T)\times \bbRD)$ and $C_p(\bbRD)$ when $d'=1$. All the notations we have just introduced can be combined in an obvious way: for example $C^1_{\lip}(\bbRD)$ will be used to denote $C^1(\bbRD)\cap\lip(\bbRD)$. We shall denote the set of Borel probability measures on $\bbRD$ with a finite first moment by $\cP_1(\bbRD)$ and distances on this set will be measured by means of the Wasserstein distance of order $1$, defined by

\bes
W_1(\mu,\mu') =\inf_{\pi\in\Pi(\mu,\mu')} \int_{\bbRD\times\bbRD}|x-x'|\pi(\De x\,\De x'),
\ees
where $\Pi(\mu,\mu')$ is the set of couplings of $\mu$ and $\mu'$, that is to say the set of probability measures on $\bbRD\times\bbRD$ whose first marginal is $\mu$ and whose second marginal is $\mu'$. Concerning matrix products and matrix-vector products, we shall employ the symbol $\cdot$. Inner product in $\bbRD$ can be equivalently be denoted $\ip{\cdot}{\cdot}$ or $\cdot$, depending on what is more convenient.

\subsection{A class of stochastic control problems}
 
Let $(\Omega,(\cF_{s})_{s\geq 0},\cF,\bbP)$ be a filtered probability space and $(B_s)_{s\geq0}$ a standard $d$-dimensional and $\cF_s$-adapted Brownian motion. Given $0\leq t<T$, we call an $\bbR^p$-valued process $(u_s)_{s\in[t,T]}$ an admissible control if $(u_s)_{s\in[t,T]}$ is progressively measurable and
\be\label{eq:mom_cond}
    \bbE\left[\int_{t}^T|u_s|^m\De s\right]<+\infty \quad \forall m\in\mathbb{N}.
\ee
Moreover, we denote $\cU_{[t,T]}$ the set of admissible controls. Next, consider a vector field $b:\bbRD\times\bbR^p\longrightarrow \bbRD$ 
satisfying Assumption \ref{ass:SP_wellposed} below. Given $\sigma>0$, a random variable $\xi$ independent from $(B_s)_{s\in[t,T]}$ and an admissible control $u$, we define $(X^{t,x,u}_s)_{s\in[t,T]}$ as the unique strong solution of the stochastic differential equation  
\begin{equation}\label{eq:contr_state}
    \begin{cases}
         \De X^{t,\xi,u}_s= b(X^{t,\xi,u}_s,u_s)\,\De s + \sigma \, \De B_s,\\ 
         X^{t,\xi,u}_t=\xi,
    \end{cases}      
\end{equation}
 This is a good definition: indeed, existence and uniqueness of a strong solution for \eqref{eq:contr_state} is proven for example in \cite[Appendix D]{fleming2006controlled}. Given functions $F:\bbRD\times\bbR^p\rightarrow \bbR$ and $g:\bbRD\rightarrow \bbR$ satisfying Assumption \ref{ass:SP_wellposed}, we study in this article the stochastic control problem
\be\label{eq:SP_prob}
\inf_{u \in \cU_{[t,T]}}  J_{t,\xi}^{T,g}(u)
\ee
where, for $u\in\cU_{[t,T]}$ the cost function is given by 
\be\label{eq:SP_obj}
 J_{t,\xi}^{T,g}(u)=\bbE\left[\int_t^T F(X^{t,\xi,u}_s,u_s)\De s + g(X^{t,\xi,u}_T)\right].
\ee
When $\xi$ has law $\delta_x$ the optimal value in \eqref{eq:SP_prob} corresponds to the value function for problem \eqref{eq:SP_prob} and we denote it $\varphi^{T,g}_t(x)$. In this case, we shall also write $X^{t,x,u}_s$ instead of $X^{t,\xi,u}_s$ in \eqref{eq:contr_state} and $J^{T,g}_{t,x}(u)$ instead of $J^{T,g}_{t,\xi}(u)$. We impose two different family of assumptions on the coefficients. We begin with the first one, that is needed to ensure existence of classical solution to the HJB equation and to establish the Hessian estimate at Theorem \ref{item_2:val_fun}. As highlighted above, none of the constants  $M_{x},M_{xx},M_{xu}$ below is used to express the multiplicative constants and ergodic rates appearing in the exponential convergence results of this article with the exception of some multiplicative constants at Theorem \ref{item_4:SP_turnpike}.

\begin{assumption}\label{ass:SP_wellposed}
We impose the following:
\ben[(i)] 
\item  $b(\cdot,\cdot)$ is of class $C^{2}_p(\bbR^{d+p};\bbRD)$. There exists $M_x\in(0,+\infty)$ such that 
\bes\label{eq:SP_wellposed}
\sup_{(x,u)\in\bbRD\times\bbR^p}|D_xb|(x,u)\leq M_x.
\ees
\item  $F(\cdot,\cdot)$ is of class $C^2_p(\bbR^{d+p})$.  Moreover,
there exist $M_{xx},M_{xu}\in(0,+\infty)$ such that
\bes
\begin{split}
\sup_{(x,u)\in\bbRD\times\bbR^p} |D_{xu} F|(x,u) \vee |D_{xu}b|(x,u) \leq M_{xu}\\
\sup_{(x,u)\in\bbRD\times\bbR^p} |D_{xx} F|(x,u) \vee |D_{xx}b|(x,u) \leq M_{xx}.
\end{split}
\ees
\een
\end{assumption}
The second requirement of item (ii) is a classical hypothesis ensuring that the value function is semiconcave. The first requirement is a somewhat less standard and we shall use it in order to establish global upper and lower bounds for the Hessian of the value function that are independent of the time-horizon by means of a coupling argument.
The second family of assumptions is the one we need to construct at Lemma \ref{lemma:SP_grad_est} a variant of coupling by reflection for controlled diffusion processes and eventually show its effectiveness in the analysis of the long time behavior of the class of stochastic control problems under consideration. At this point, following \cite{lindvall1986coupling}, it is convenient to introduce for any $u\in\bbRD$ the function $ \kappa_{b(\cdot,u)}: (0,+\infty)\longrightarrow\bbR$ as follows: 
\be\label{eq:eberle_kappa}
    \kappa_{b(\cdot,u)}(r) = \inf\left\{-\frac{2\ip{b(x,u)-b(x',u)}{x-x'}}{\sigma^2|x-x'|^2}: |x-x'|=r  \right\}, \quad \bar \kappa_b=\inf_{u\in\bbR^p}\kappa_{b(\cdot,u)}.
\ee
As we are about to see, we express our mild dissipativity assumptions on $b$ through  $\bar\kappa_{b}$.
\begin{assumption}\label{ass:SP}
We impose the following:
\ben[(i)]
\item The function $\bar\kappa_b$  satisfies 
            \be\label{eq:SP_drift_ass_1}
                \liminf_{r\rightarrow+\infty} \bar\kappa_{b}(r)>0, \quad \int_{0}^1r\,\big(\bar\kappa_b^{-}\big)(r)\De r <+\infty,
            \ee
            where $\bar\kappa^{-}_b(r)=\max\{-\bar\kappa_b(r),0\}$.
           \item There exists $M_u\in(0,+\infty)$ such that
            \be\label{eq:SP_drift_ass_2}
              |D_u b|(x,u) \leq M_u,\quad  |\partial_u F|(x,0) \leq M_u,\quad \forall x\in\bbRD,u\in\bbR^p            
            \ee
\item There exist finite positive constants $M^{F}_{x},M^g_x\in(0,+\infty)$ such that 
        \begin{equation}\label{eq:SP_grad_to_control_ass}
            \begin{split}
                |F(x,u)-F(x',u)|&\leq M^{F}_{x}|x-x'|, \quad \forall \,x,x'\in\bbRD,u\in\bbR^p,\\
                |g(x)-g(x')|&\leq M^g_x|x-x'|, \quad \forall x,x'\in\bbRD.
            \end{split}
        \end{equation}
\item For any $R>0$ there exist $\omega_R\in(0,+\infty)$ such that 
 such that 
        \be\label{eq:SP_Fconv_ass}
            \partial_{uu}F(x,u)+ D_{uu}b(x,u)\cdot p \succeq \omega^2_R I \quad \forall x\in\bbRD,u\in\bbR^p,\, |p|\leq R,
        \ee
    where the above inequality is understood as an inequality between quadratic forms. In particular, for any $(x,p)\in\bbRD\times\bbRD$, the function
        \be\label{eq:SP_conv_ass}
   \bbR^p\ni u\mapsto F(x,u)+b(x,u)\cdot p 
\ee
admits a unique minimizer $w(x,p)$.

        %
        
\een
\end{assumption}

When $b$ does not depend on $u$, assumptions of the form \eqref{eq:SP_drift_ass_1} are standard in applications of coupling by reflection. Likewise, \eqref{eq:SP_drift_ass_2} is a rather classical assumption whereas the Lipschitzianity requirements \eqref{eq:SP_grad_to_control_ass}  represent pretty common assumptions in stochastic control, see the textbooks \cite{yong1999stochastic,fleming2006controlled}, and are very often encountered in the study of ergodic stochastic control problems, see for example \cite{fuhrman2009ergodic,debussche2011ergodic,hu2019ergodic}. However, imposing this assumption leaves out some interesting situations that need to be addressed separately. It is nevertheless a sharp assumption, in the sense that if we drop it, there is no reason to expect uniform in time gradient estimates  akin to Theorem \ref{thm:SP_val_fun} to hold.
We are ready to state our first main result on the behavior of the value function that, as it is well known, is a candidate solution for the HJB equation 
\begin{equation}\label{eq:SP_HJB}
    \begin{cases}
    \partial_t \varphi_t(x) -H(x,\nabla \varphi_t(x))+\frac{\sigma^2}{2}\Delta\varphi_t(x)=0 \quad (t,x) \in(0,T)\times \bbRD,\\
    \varphi_T(x)=g(x) \quad x\in\bbRD
    \end{cases}
\end{equation}
In the above, the Hamiltonian $H:\bbRD \times \bbRD \rightarrow\bbR$ is defined as usual by
\bes
H(x,p)= -\inf_{u\in\bbR^p}\{ F(x,u)+b(x,u)\cdot p \}.
\ees
 \subsection{Large-time behavior of the value function}\label{sub:val_fun}
 We now report on our main results about the large time behavior of the value function. The first two contributions of Theorem \ref{thm:SP_val_fun} are a uniform in time gradient estimate under a weak dissipativity condition on the drift field and a uniform in time Hessian estimate. Then, we show existence of a unique stationary viscosity solution for \eqref{eq:SP_HJB} as well as an exponential contraction result. We employ the following definition of stationary solution:  a pair $(\alpha^{\infty},\varphi^{\infty})$ in $\bbR \times  C(\bbRD)$ such that $\varphi^{\infty}(0)=0$ is a stationary viscosity solution for \eqref{eq:SP_HJB} if and only if for any $T>0$ the function
 \bes
[0,T)\times\bbRD\ni(t,x)\mapsto \alpha^{\infty}(T-t)+\varphi^{\infty}(x)
 \ees
 is a viscosity solution for \eqref{eq:SP_HJB} according to the classical definition, see e.g. \cite[Def 5.1]{yong1999stochastic}. The last contribution is an exponential convergence result towards the stationary solution. All constants appearing in the estimates below are explicit and we provide precise references to where in the article their full form is given. We do not write them explicitly at this stage since doing so requires some notation related to coupling by reflection that will be introduced at Section \ref{sec:coup_by_ref}.
 
\begin{theorem}\label{thm:SP_val_fun}
Let Assumption \ref{ass:SP_wellposed}-\ref{ass:SP} hold.  
\ben[(i), ref=\thetheorem(\roman*)]
 \item\label{item_1:val_fun} Let $g\in\mathrm{Lip}(\bbRD)$. Then we have
\bes
\forall 0\leq t\leq T \quad \| \varphi^{T,g}_t\|_{\lip}\leq M^{\varphi,g}_{x,T-t}, \quad \sup_{0\leq t\leq T} \|\varphi^{T,g}_t\|_{\lip} \leq M^{\varphi,g}_x.
\ees
where $M^{\varphi,g}_{x,T-t},M^{\varphi,g}_x\in(0,+\infty)$ are given by \eqref{eq:SP_grad_est_3} and \eqref{eq:SP_grad_est_7} respectively and are independent of $M_{x},M_{xx},M_{xu}$.
\item\label{item_2:val_fun} Let $g\in\mathrm{Lip}(\bbRD)$ and $M^{\varphi,g}_{xx,T-t}$ be given by \eqref{eq:SP_Hess_est_2}. Then $\varphi^{T,g}_t\in C^1(\bbRD)$ for any $t<T$ and we have
\bes
\|\nabla \varphi^{T,g}_t\|_{\lip}\leq M^{\varphi,g}_{xx,T-t}, \quad \sup_{ 0\leq t\leq T-\varepsilon} \|\nabla \varphi^{T,g}_t\|_{\lip}\leq M^{\varphi,g}_{xx,\varepsilon} \quad \forall\varepsilon>0.
\ees
\item\label{item_3:SP_val_fun} There exist a unique pair $(\alpha^{\infty},\varphi^{\infty})$ in $\bbR \times  C^1_{\mathrm{Lip}}(\bbRD)$ such that $(\alpha^{\infty},\varphi^{\infty})$ is a stationary viscosity solution for \eqref{eq:SP_HJB}. Moreover, $\|\varphi^{\infty} \|_{\mathrm{Lip}}\leq M^{\varphi,0}_{x}$ and $\nabla\varphi^{\infty}\in{\mathrm{Lip}}(\bbRD;\bbRD)$.
\item\label{item_4:SP_val_fun}  Let $g,g'\in C^1_{\lip}(\bbRD)$. Then
\bes
\forall 0\leq t \leq T, \quad \|\varphi^{T,g}_t-\varphi^{T,g'}_t \|_{\mathrm{Lip}} \leq C^{-1}\| g-g' \|_{\mathrm{Lip}} \exp(-\lambda (T-t)),
\ees
where $C,\lambda\in(0,+\infty)$ are as in \eqref{eq:SP_stab_fin_cond_1} and in particular independent of $M_x,M_{xu},M_{xx}$. As a consequence, for any $g'\in C^1_{\lip}(\bbRD)$:
\bes
\forall 0\leq t \leq T, \quad\|\varphi^{T,g'}_t-\varphi^{\infty} \|_{\mathrm{Lip}} \leq C^{-1}\| g'- \varphi^{\infty} \|_{\mathrm{Lip}} \exp(-\lambda (T-t)).
\ees
\een
\end{theorem}
Uniform in time Lipschitz estimates similar to \ref{item_1:val_fun} for non-linear PDEs have been obtained in \cite{porretta2013global} with a method that can indeed be seen as an analytical equivalent of coupling by reflection, as explained in \cite[Appendix A]{porretta2013global}. There are very few Hessian bounds akin to Theorem \ref{item_2:val_fun} in the literature that are valid under hypothesis comparable to ours, see \cite{fujita2005hessian} for some very precise calculations in the framework of controlled Ornstein-Uhlenbeck processes.

\subsection{An exponential turnpike estimate}\label{sub:turnpike}
We now turn the attention to optimal processes and prove an exponential turnpike theorem in Wasserstein distance of order one.  Using a classical verification argument, that we detail at Proposition \ref{prop:SP_opt_cond} and \ref{prop:policy_opt}, an optimal control for \eqref{eq:SP_prob} is given by the process
\bes
\bmu^{0,\xi,T,g}_s= w(\bmX^{0,\xi,T,g}_s,\nabla\varphi^{T,g}_s(\bmX^{0,\xi,T,g}_s)),\quad s\in[0,T],
\ees 
where $(\bmX^{0,\xi,T,g}_s)_{s\in[0,T]}$ is the unique solution  of the stochastic differential equation (SDE)
\be\label{eq:Ham_SDE}
\begin{cases}
\De X_s= - D_pH(X_s, \nabla\varphi^{T,g}_s(X_s))\De s +\sigma \De B_s,\\
X_0=\xi.
\end{cases}
\ee
and $w(\cdot,\cdot)$ has been defined within Assumption \ref{ass:SP}. The main question we address is to quantify the speed of convergence of $\bmX^{0,\xi,T,g}_s$ and $\bmu^{0,x,T,g}_s$ towards their ergodic limits. In this context, the turnpike is defined as the unique stationary distribution $\mu^{\infty}$ associated with the ergodic drift field obtained replacing $\nabla\varphi^{T,g}_s$ with $\nabla\varphi^{\infty}$ in the \eqref{eq:Ham_SDE}. The law of the ergodic optimal control, that we denote $\nu^{\infty}$, is obtained as the image measure of {$w(\cdot,\nabla\varphi^{\infty}(\cdot))$} through $\mu^{\infty}$.
 
\begin{theorem}\label{thm:SP_turnpike_1}
Let Assumption \ref{ass:SP_wellposed}-\ref{ass:SP} hold.  
\ben[(i), ref=\thetheorem(\roman*)]
\item\label{item_1:SP_turnpike_1}Let $g,g'\in C^1_{\lip}(\bbRD)$. Then the estimate
\be\label{eq:SP_turnpike_8}
\begin{split}
W_1(\mathrm{Law}(\op_s),\mathrm{Law}(\bmX^{0,\xi',T,g'}_s)) \leq C\, W_1(\mathrm{Law}(\xi),\mathrm{Law}(\xi'))\exp(-\lambda s)\\
+ C\,\|g-g'\|_{\mathrm{Lip}} \exp(-\lambda (T-s)\big)
\end{split}
\ee
holds uniformly on $\xi,\xi'\in\cP_1(\bbRD)$ and $0\leq s\leq T$. Moreover,  $C=\max\{\overrightarrow{C},\overleftarrow{C}\}$ $\lambda=\min\{\overrightarrow{\lambda},\overleftarrow{\lambda}\}$, where $(\overrightarrow{C},\overrightarrow{\lambda})$ are defined at \eqref{eq:contraction_SP_5} and $(\overleftarrow{C},\overleftarrow{\lambda})$ at \eqref{eq:sticky_coup_4}.
\item\label{item_2:SP_turnpike_1} Let $\varphi^{\infty}$ be as in Theorem \ref{item_3:SP_val_fun}. Then the stochastic differential equation
\be\label{eq:SP_turnpike_2}
\begin{cases}
\De X_s= - D_pH(X_s, \nabla\varphi^{\infty}(X_s))\De s +\sigma \,\De B_s,\\
X_0=x.
\end{cases}
\end{equation}
admits a pathwise unique strong solution and a unique stationary distribution $\mu^{\infty}\in\cP_1(\bbRD)$ for \eqref{eq:SP_turnpike_2}.
\item\label{item_3:SP_turnpike} There exists a constant $\lambda^{\infty}>0$ such that for all $g'\in C^1_{\lip}(\bbRD)$ there exist constants $\tau,A$ such that the estimate
\be\label{eq:SP_turnpike_4}
\begin{split}
W_1(\mathrm{Law}(\bmX^{0,\xi',T,g'}_s),\mu^{\infty}) \leq A\Big(\, W_1(\mathrm{Law}(\xi'),\mu^{\infty})\exp(-\lambda^{\infty} s)\\
+\|g'-\varphi^{\infty}\|_{\mathrm{Lip}} \exp(-\lambda^{\infty}(T-s))\Big)
\end{split}
\ee
holds uniformly on $\mathrm{Law}(\xi')\in\cP_1(\bbRD)$ and $0\leq s\leq T-\tau$. The precise form of $\lambda^{\infty}$ is given at \eqref{eq:SP_turnpike_15} and $\tau$ is defined at \eqref{eq:SP_turnpike_13}. In particular, these constants as well as $A$ are independent of $M_{xu},M_{xx},M_x$. Moreover, $\tau$ and $A$ depend on $g'$ only through $\|g'\|_{\lip}$.
\item\label{item_4:SP_turnpike} Let $\nu^{\infty}$ be the image measure of {$w(\cdot,\nabla\varphi^{\infty}(\cdot))$} under $\mu^{\infty}$ and $\lambda^{\infty}$ be as before. Then, for all $g'\in C^1_{\lip}(\bbRD)$ there exist constants $\tau,A$ such that the estimate
\be\label{eq:SP_turnpike_21}
\begin{split}
W_1(\mathrm{Law}(\bmu^{0,\xi',T,g'}_s),\nu^{\infty}) \leq A\Big(\, W_1(\mathrm{Law}(\xi'),\mu^{\infty})\exp(-\lambda^{\infty} s)\\
+\|g'-\varphi^{\infty}\|_{\mathrm{Lip}} \exp(-\lambda^{\infty}(T-s))\Big)
\end{split}
\ee
holds uniformly on $\mathrm{Law}(\xi')\in\cP_1(\bbRD)$ and $0\leq s\leq T-\tau$.
\een
\end{theorem}

\begin{remark}
\bei 
\item An important difference between the constant $\lambda$ appearing at \eqref{eq:SP_turnpike_8} and $\lambda^{\infty}$ appearing at \eqref{eq:SP_turnpike_4} is that the latter does not depend on $g'$, whereas the first does.
\item  The constant $A$ appearing at item $(iii)$ admits an explicit expression that we report at \eqref{eq:SP_turnpike_16}. However, such expression is more involved than the one for $\lambda^{\infty}$ and $\tau$. 
\item In contrast with all other estimates, the multiplicative constant $A$ appearing at item \ref{item_4:SP_turnpike} does depend on $M_{xx},M_{xu},M_x$. However, the exponential rate $\lambda^{\infty}$ does not. The constant $\tau$ can be taken to be the same as in item \ref{item_3:SP_turnpike}. 
\eei
\end{remark}
According to the standard terminology, we say that $\mu^{\infty}$ is a stationary distribution for \eqref{eq:SP_turnpike_2} if  
\bes
\int_{\bbRD}\bbE[f(X^{x}_t)]\mu^{\infty}(\De x)= \int_{\bbRD}f(x)\mu^{\infty}(\De x), \quad \forall t>0.
\ees
for all bounded and measurable functions, where $(X^x_s)_{s\geq0}$ is a solution of \eqref{eq:SP_turnpike_2} with initial condition $X^x_0=x.$
\begin{remark}
The exponential rates in the above turnpike estimates can be improved using the Hessian bound of Theorem \ref{item_2:val_fun}. However, we prefer to keep them in the present form to highlight once more that, at least in principle, exponential ergodicity requires only Assumption \ref{ass:SP}.
\end{remark}
We present here two simple examples where Theorem \ref{thm:SP_val_fun}. and \ref{thm:SP_turnpike_1} can be applied. We insist one last time on the fact that not all assumptions are necessary for the ergodicity results.
\begin{myxmpl}\label{ex:FK}
A simple setting where our main results can be applied is obtained considering 
\bes
b(x,u)=b_0(x) +u,\quad F(x,u)= \ell(|u|) + f(x) 
\ees
where
\bei 
\item $g\in C^1_{\lip}(\bbRD)$, $f\in C^2_{\lip}(\bbRD)$ with bounded second derivative 
\item $\ell$ is a uniformly convex smooth function with polynomial growth. 
\item $b_0\in C^2_{\lip}(\bbRD;\bbRD)$ is a vector field with bounded second derivative satisfying 
\bes
-\ip{b_0(x)-b_0(x')}{x-x'}\geq \begin{cases} \kappa|x-x'|^2, \quad  |x-x'|\geq R,\\
                                            -L|x-x'|^2, \quad |x-x'|\leq R.\end{cases}
\ees
for some $R,\kappa,L\in(0,+\infty)$. This condition is fulfilled for example when 
\bes b_0(x)=-\alpha x+\gamma(x)
\ees
with $\alpha>0$ and $\gamma$ a bounded smooth vector field with bounded derivative
\eei
\end{myxmpl}

\begin{myxmpl}
A more general setting where our main results can be applied is obtained considering 
\bes
b(x,u)=b_0(x) +u,\quad F(x,u)=\ell(|u|)+ L(x,u) + f(x) 
\ees
where $b_0,f,g,\ell$ are as in Example \ref{ex:FK} and 
\bei
\item $\sup_{(x,u)\in\bbR^{d+p}}|D_x L(x,u)|<+\infty$
\item $\sup_{(x,u)\in\bbR^{d+p}}|D_{uu}L(x,u)|\leq \delta$, and $\delta$ is small enough.
\item $\sup_{(x,u)\in\bbR^{d+p}}|D_{xx} L(x,u)|+|D_{xu} L(x,u)|<+\infty$
\eei
\end{myxmpl}

\paragraph{Stochastic and deterministic control} It is interesting to compare the findings of Theorem \ref{thm:SP_val_fun} and \ref{thm:SP_turnpike_1} with recent developments around the study of the turnpike property in the field of deterministic control. In what concerns the results, we obtain here global estimates, whereas most of the deterministic literature seems to focus on local estimates, in which a smallness condition is imposed on boundary data so that they are, in a sense, close enough to the turnpike solution. This is the case for example in the influential work \cite{trelat2015turnpike}. There are however some notable exceptions. Indeed, in \cite[Thm.1]{trelat2020linear} the author shows how to pass from local to global estimates if one is able to construct a so called storage function. On another note, the authors of \cite{esteve2022turnpike} manage to establish an exponential turnpike result for a deterministic problem with quadratic cost function and control-affine dynamics. At the level of assumptions, there seems to be a structural difference between the stream of works issued from \cite{trelat2015turnpike} and the present setting. 
In this article, where the cost function is at most linear in the space variable, the emergence of the exponential turnpike property is due to the fact that the drift field satisfies some kind of uniform ergodicity conditions, that in our case is \eqref{eq:SP_drift_ass_1}. The same can be said about several other works devoted to stochastic control, see  \cite{fujita2006asymptotic,debussche2011ergodic,hu2015probabilistic,cosso2016long}: we refer to the end of this section for a more detailed comparison with this line of research. On the contrary, in \cite{trelat2015turnpike} and in most related works, there is generally no ergodicity assumption on the drift and the key geometric feature of the system triggering the exponential turnpike phenomenon is that the cost function includes a strongly confining term in the space variable, typically a quadratic term, see for example \cite{sakamoto2021turnpike}. This kind of situation is natural in stochastic control as well and deserves to be studied in depth, see \cite{ichihara2012large} for some results in this direction. In this setting, instead of looking for Lipschitz bounds on $\varphi^{T,g}_t$, that cannot hold, a possible strategy to obtain global turnpike estimates is to establish one-sided Lipschitz bounds for $-D_pH(\cdot,\nabla\varphi^{T,g}_{\cdot}(\cdot))$: coupling methods seem to be an effective tool to accomplish this task, as we shall report in forthcoming work. One of the main motivation for writing this article was to show that moving from  the deterministic to the stochastic setting makes uniform and global estimates more accessible. One can see this happening already in the relatively simple setting of Example \ref{ex:FK}. Indeed, consider the situation where $b_0=-\nabla U$ and the potential $U$ has several global minima. To the best of our knowledge, neither global turnpike estimates nor uniform in time gradient estimates for the corresponding first order Hamilton Jacobi equation are known to hold for the noiseless version $(\sigma=0)$ of problem \eqref{eq:SP_prob}. Finally, we remark that another difference between the deterministic and stochastic control literature seems to be that the former is mostly studying the behavior of optimal trajectories whereas the latter is essentially focused on the Hamilton Jacobi Bellman equation. The recent article \cite{esteve2020turnpikeb} offers a unified perspective by analysing the effect of the turnpike property on the Hamilton Jacobi Bellman equation in the framework of deterministic linear-quadratic control.

\paragraph{More on controlled diffusion processes} The article \cite{arisawa1998ergodic} undertakes a systematic study of ergodic stochastic control in a periodic setting, including situations where control variables enter the diffusion matrix. The periodic setup marks a fundamental difference with the current framework that is closer to the one in \cite{fujita2006asymptotic} that is concerned with optimal control of Ornstein-Uhlenbeck processes. In there, the Hamiltonian  $H(x,\nabla\varphi(x))$ is supposed to be of the form $h(\nabla \varphi(x))+f(x)$ and local uniform convergence of the value function is proven under two different sets of assumptions. In the first scenario, $f$ is H\"older-continuous and $h$ is Lipschitz. In the second case, $f$ is globally Lipschitz and $h$ locally Lipschitz. It is worth noticing that the proof strategy in this case rests on a uniform in time gradient estimate (Eq 6.2) akin to Theorem \ref{item_1:val_fun} that can be obtained by means of synchronous coupling. The stream of works initiated by \cite{fuhrman2009ergodic}\cite{debussche2011ergodic}, which includes \cite{cosso2016long}, introduces a more probabilistic viewpoint by focusing on the representation of the value function through backward stochastic differential equations (BSDEs) and including infinite dimensional SDEs in the analysis. In particular, \cite{debussche2011ergodic} overcomes the strong dissipativity assumptions made in \cite{fuhrman2009ergodic} allowing for bounded Lipschitz perturbations  of strictly dissipative operators, such as the Ornstein-Uhlenbeck operator, in the controlled dynamics. A step forward is taken in \cite{hu2015probabilistic} where an exponential rate of convergence towards the ergodic BSDE is provided (see \cite[Thm 4.4]{hu2015probabilistic}, possibly in an infinite-dimensional setting. When reduced to the finite dimensional setting the assumptions made there to show exponential convergence would require, among other things, that 
\be\label{eq:comp_lit_1}b(x,u) = -A x+ R(u)\ee 
with $A$ a strictly positive matrix and $R$ a bounded vector field. These assumptions are more restrictive than ours. On the other hand, the assumptions made on the running cost allow for polynomial growth in the space variable, that violates \eqref{eq:SP_grad_to_control_ass} above. The subsequent work \cite{hu2019ergodic} improves on \cite{hu2015probabilistic} in two directions. Firstly, multiplicative noise is allowed, i.e. $\sigma$ can be taken to be a function of $X^{t,x,u}_s$ in \eqref{eq:contr_state}, a situation not considered in this work. About this issue, we mention that following \cite[Rem. 2]{eberle2016reflection} an extension of our main results to the multiplicative setting could in principle be possible. Secondly, assumption \eqref{eq:comp_lit_1} is modified replacing $-Ax$ with a weakly dissipative vector field. Moreover, the Lagrangian function is assumed to be Lipschitz with respect to the space variable. These assumptions bear many resemblances with  \eqref{eq:SP_drift_ass_1} and \eqref{eq:SP_grad_to_control_ass}, though some important differences remain; in particular, the boundedness of the function $R(\cdot)$ in \eqref{eq:comp_lit_1} is  still assumed in \cite{hu2019ergodic}. We conclude this overview recalling that the long time convergence of the value function is also obtained in \cite{ichihara2012large,ichihara2013large} under different assumptions. Essentially, they remove all kind of dissipative structure in the controlled dynamics, i.e. they consider $b(x,u)=u$ in \eqref{eq:contr_state} and replace it with suitable coercivity assumptions both in the space and control variables on the cost function $F(\cdot,\cdot)$. In particular, they are able to treat a class of problems  where $F(x,u)$  has polynomial growth in the space variable. A common point of all works discussed in this subsection is that the asymptotic convergence results with exponential rate concern the value function and require some kind of strong dissipativity assumptions, with the exception of \cite{hu2019ergodic} that only assume a weak dissipativity condition. Convergence of gradients is also shown without rate in some cases. To the best of our understanding, exponential convergence of the gradient of the value function, the turnpike property of optimal controls and optimal process as well as the uniform Hessian bounds are novel contributions, at least under the weak dissipativity assumption \eqref{eq:SP_drift_ass_1}. At the proofs level, as we said before, one key ingredient is to introduce and analyse a controlled version of coupling by reflection and to apply systematically sticky coupling and coupling by reflection for uncontrolled diffusion processes to establish the main estimates. Interestingly, coupling arguments are also at the heart of the proof strategy in \cite{debussche2011ergodic,hu2015probabilistic,hu2019ergodic}. A first difference with these works consist in the different kind of couplings considered. Here, we mostly consider couplings implying contraction estimates in Wasserstein distance, whereas there estimates in total variation distance play a fundamental role. Another different aspect worth mentioning is that we construct couplings of controlled diffusion processes (see Lemma \ref{lemma:SP_grad_est}), whereas the authors of the above listed articles manage to establish their results leveraging the ergodic properties of couplings for classical non-controlled diffusion processes. Finally let us briefly discuss some analogies with recently obtained results on the turnpike property for mean field games, see \cite{cardaliaguet2012long,cardaliaguet2013long}. In this context, the exponential convergence to equilibrium is triggered by a well known monotonicity condition on the interaction between players. This geometric condition can be viewed to some extent as a replacement for the various dissipativity conditions encountered so far. In \cite{cirant2021long}, the authors manage to relax the strict monotonicity assumption and write that they can do so because ``the Brownian noise in the individual dynamics can compensate, to some extent, the lack of monotonicity". This affirmation draws an interesting parallelism with coupling by reflection, whose success in applications is essentially due to the same reason: the presence of Brownian motion allows to compensate short-range convexity deficits in the drift of a diffusion process.

\paragraph{Organization of the article} In Section \ref{sec:coup_by_ref} we give some preliminaries on coupling by reflection and sticky coupling. Section \ref{sec:proofs} is devoted to the proof of our main results. The Appendix section contains the proof of the more technical statements.

\section{Preliminaries on coupling by reflection and sticky coupling}\label{sec:coup_by_ref}

This section is a summary of known results, mostly taken from \cite{eberle2016reflection},\cite{eberle2019sticky} on coupling by reflection and sticky coupling upon which the proofs of the next section are built. 

\subsection{Coupling by reflection} 
Given a filtered probability space $(\Omega,\cF,(\cF_s)_{s\geq0},\bbP)$ supporting an $\cF_s$-adapted Brownian motion $(B_s)_{s\geq0}$, the coupling by reflection of two solutions $(X^{x}_s)_{s\in[0,T]},(X^{x'}_s)_{s\in[0,T]}$ of the SDE

\be\label{eq:SDE_1}
\De X_s = \beta_s(X_s)\De s +\sigma \De B_s 
\ee
with initial conditions $x$ and $x'$ respectively is the stochastic process $(\bar X^x_s,\bar X^{x'}_s)_{s\in[0,T]}$ defined by $\bar{X}^{x}_0=x,\bar{X}^{x'}_0=x'$ and
\be\label{eq:coup_by_ref}
\begin{cases}
\De \bar{X}^{x}_s = \beta_s(\bar X^{x}_s)+\sigma \De B_s \quad &\text{for}\, 0\leq s< \tau,\\
\De \bar X^{x'}_s = \beta_s(\bar X^{x'}_s)+\sigma \De \check{B}_s, &\text{for}\, 0\leq s\leq \tau,\, \bar X^{x'}_s=\bar X^{x}_s \quad \text{for}\, s\geq \tau,
\end{cases}
\ee
where 
\bes
\De \check{B}_s= (\mathrm{I}-2 \,\rme_s \cdot \rme^{\top}_s\mathbf{1}_{\{s<\tau\}})\cdot \De B_s, \quad \tau=\inf\{s:\bar X^{x'}_s=\bar X^{x}_s \},
\ees
$\mathrm{I}$ is the identity matrix and
\bes
\rme_s = \frac{\bar{X}^{x}_s-\bar{X}^{x'}_s}{|\bar{X}^{x}_s-\bar{X}^{x'}_s|}, \quad \text{for}\, 0\leq s<\tau .
\ees
As in Assumption \ref{ass:SP}, it is convenient, for a given vector field $\beta$,  to define the function $\kappa_{\beta}:[0,+\infty)\longrightarrow\bbR$ by
\bes
\kappa_{\beta}(r)=\inf\left\{-2\frac{\ip{\beta(x)-\beta(x')}{x-x'}}{\sigma^2|x-x'|^2}: |x-x'|=r \right\}, \quad r>0.
\ees
This function is used to give hypothesis under which \eqref{eq:coup_by_ref} is well posed and to express the main properties of coupling by reflection. When $\beta=-\nabla U$, $\kappa_{\beta}(r)$ can be interpreted as a kind of integrated modulus of convexity for $U$ on intervals of length $r$ . Indeed, one can easy verify that 
\bes
\frac{\sigma^2}{2}\kappa_{\beta}(r) = \inf_{x\in\bbRD, |v|=1}\left\{ \frac{1}{r}\int_{0}^r\ip{\nabla^2U(x+\theta v)\cdot v}{v}\De \theta \right\}
\ees

We now introduce some notation, following verbatim \cite{eberle2016reflection}. We begin by defining the set $K$ as follows
\be\label{eq:Lyapunov_funct_coup_by_ref_1}
K=\left\{ \kappa\in C((0,+\infty);\bbR) \quad \text{s.t.}\quad \liminf_{r\rightarrow+\infty}\kappa(r)>0, \quad \int_0^1r\kappa^{-}(r)\De r<+\infty\right\}.
\ee
Next, we introduce a twisted version of $W_1(\cdot,\cdot)$ needed to express the contractive properties of coupling by reflection. We begin by defining for $\kappa\in K$ the quantities
\be\label{eq:aux_fun_eberle_2}
\begin{split}
    R_0=\inf\{R\geq 0: \kappa(r)\geq 0 \,\,\forall r\geq R\},\\
    R_1=\inf\{ R\geq R_0: \kappa(r)R(R-R_0)\geq 8\,\, \forall r \geq R\}.
\end{split}
\ee
Moreover, we introduce auxiliary functions $\phi,\Phi,g$ as follows
\be\label{eq:aux_fun_eberle_1}
\begin{split}
\phi(r)=\exp\Big(-\frac14\int_0^r s\kappa^-(s)\De s\Big),\quad \Phi(r)=\int_0^r\phi(s)\De s,\\
g(r)=1- \frac{\int_0^{r \wedge R_1}\Phi(s)/\phi(s) \De s}{2\int_0^{R_1}\Phi(s)/\phi(s) \De s}.
\end{split}
\ee

\begin{mydef}\label{def:Lyapunov_funct_coup_by_ref}
Let $K$ be defined by \eqref{eq:Lyapunov_funct_coup_by_ref_1} and $F$ be defined by
\bes
F= \{ f\in C^2((0,+\infty);(0,+\infty)) \quad \text{s.t.}\quad f'(0)=1,f'(r)>0,f''(r)\leq 0 \quad \forall r>0\}.
\ees
We define maps 
\bes
\bmf:K\longrightarrow F, \quad \bm\lambda:K\longrightarrow(0,+\infty),\quad \bmC:K\longrightarrow(0,1]
\ees
as follows
\begin{subequations}
\bes
\bmf(\kappa)(r)=\int_0^r\phi(s)g(s)\De s, \quad \forall r>0,\\
\ees
\bes
\bm\lambda^{-1}(\kappa)=\sigma^{-2} \int_0^{R_1}\Phi(s)/\phi(s)\De s, 
\ees
\bes
\bmC(\kappa)=\min\left\{ \frac{\phi(R_0)}{2}, \Big(2\int_0^{R_1}1/\phi(s)\De s\Big)^{-1},\bm\lambda(\kappa)\Phi(R_1)\right\}
\ees
\end{subequations}
  where for any $\kappa\in K$, $R_0,R_1$ are defined at \eqref{eq:aux_fun_eberle_2} and the functions $\phi,\Phi,g$ are defined at \eqref{eq:aux_fun_eberle_1}.

 \end{mydef}
 
 \begin{prop}\label{prop:Lyapunov_funct_coup_by_ref}
 Let $\kappa\in K$ and $(f,\lambda,C)=(\bmf(\kappa),\bm\lambda(\kappa),\bmC(\kappa))$. Then the following hold
 \ben[(i), ref=\theprop(\roman*)]
 \item\label{item_1:Lyapunov_funct_coup_by_ref} $f$ is equivalent to the identity:
 \be\label{eq:Lyapunov_funct_coup_by_ref_2}
 C\, r\leq f(r) \leq r, \quad C\leq f'(r) \leq 1, \quad \forall r>0.
\ee
\item\label{item_2:Lyapunov_funct_coup_by_ref} The differential inequality
\be\label{eq:Eberle_funct_2}
f''(r)-\frac{r}{4}\kappa(r)f'(r)\leq -\frac{\lambda}{2\sigma^{2}} f(r),
\ee
holds for all $r>0$.
\item\label{item_3:Lyapunov_funct_coup_by_ref}  The maps $\bm\lambda$ and $\bmC$ are monotone in the following sense: if $\kappa,\kappa'\in K$ are such that 
 \bes
\kappa(r)\geq\kappa'(r) \quad \forall r>0,
 \ees
 then 
 \bes
 \bm\lambda(\kappa)\geq\bm\lambda(\kappa'), \quad \bmC(\kappa)\geq\bmC(\kappa').
 \ees
 and 
 \bes
\bmf(\kappa')''(r)-\frac{r}{4}\kappa(r)\bmf(\kappa')'(r)\leq -\frac{\bm\lambda(\kappa')}{2\sigma^{2}} \bmf(\kappa')(r), \quad \forall r>0.
 \ees

 \een
 \end{prop}
\begin{proof}
The proof of item $(i)$ and $(ii)$ is carried out within the proof of \cite[Th. 1]{eberle2016reflection}. In fact, this result is stronger than the one we are reporting, as the constant $\bmC(\kappa)$ can be taken there to be $\phi(R_0)/2$. However, we use here the suboptimal form of Definition \ref{def:Lyapunov_funct_coup_by_ref} for later convenience. Item (iii) is proven by means of elementary calculations. Indeed, it is easily seen that $\phi$ is monotonically increasing in $\kappa$ and $R_0,R_1$ are monotonically decreasing. From this, the monotonicity of $\bmC$ follows at once. The monotonicity of $\bm\lambda$ is also readily obtained from the following alternative definition, also given in \cite{eberle2016reflection}:
\bes
\bm\lambda(\kappa)^{-1} = \frac{1}{\sigma^2} \int_0^{R_1}\int_0^s \exp \Big( \frac14\int_t^su\kappa^{-}(u)\De u\Big)\De t \De s.
\ees
\end{proof}
We can now report on the contractive properties of coupling by reflection. To do so, for any $\kappa\in K$ we introduce the following twisted version of the Wasserstein distance
\bes
W_f(\mu,\mu') = \inf_{\pi\in\Pi(\mu,\mu')} \int f(|x-x'|)\pi(\De x \De x'), \quad f = \bmf(\kappa).
\ees
Note that because of Proposition \ref{item_1:Lyapunov_funct_coup_by_ref} $W_f$ is equivalent to $W_1$.
\begin{prop}\label{prop:contraction_coup_by_ref}
Fix $T>0$ and assume that $\beta:[0,T]\times\bbRD\longrightarrow \bbRD$ is locally Lipschitz continuous. Moreover, assume that
\bes
\bar{\kappa}_{\beta}:=\inf_{s\in[0,T]}\kappa_{\beta_s }\in K 
\ees
and set $(f,\lambda,C)=(\bmf(\bar\kappa_\beta),\bm\lambda(\bar\kappa_{\beta}),\bmC(\bar\kappa_{\beta}))$. Then the following hold
\ben[(i), ref=\theprop(\roman*)]
\item\label{item_1:contraction_coup_by_ref} For any $\xi$ with $\mathrm{Law}(\xi)\in\cP_1(\bbRD)$, \eqref{eq:SDE_1} admits a strong solution $(X^{\xi}_s)_{s\in[0,T]}$ and pathwise uniqueness holds. For any $x,x'\in\bbRD$  \eqref{eq:coup_by_ref} admits a strong solution $(\bar{X}^{x}_s,\bar{X}^{x'}_s)_{s\in[0,T]}$ such that  \bes 
\mathrm{Law}\big((\bar{X}^{y}_s)_{s\in[0,T]}\big)=\mathrm{Law}\big((X^{y}_s)_{s\in[0,T]}\big), \quad y=x,x'.
\ees 
In other words, $(\bar{X}^{x}_s,\bar{X}^{x'}_s)_{s\in[0,T]}$ is a coupling of $(X^{x}_s)_{s\in[0,T]}$ and $(X^{x'}_s)_{s\in[0,T]}$.
\item\label{item_2:contraction_coup_by_ref} The contraction estimate
\bes
\bbE[|\bar{X}^{x}_s-\bar{X}^{x'}_s|]\leq C^{-1}|x-x'|\exp(-\lambda s) 
\ees
holds uniformly on $s\in[0,T],x,x'\in\bbRD$. As a consequence, we have that
\be\label{eq:contraction_coup_by_ref_3}
\begin{split}
W_f(\mathrm{Law}(X^{\xi}_s),\mathrm{Law}(X^{\xi'}_s)) \leq \exp(-\lambda s)W_f(\mathrm{Law}(\xi),\mathrm{Law}(\xi'))\\
W_1(\mathrm{Law}(X^{\xi}_s),\mathrm{Law}(X^{\xi'}_s)) \leq C^{-1}\exp(-\lambda s)W_1(\mathrm{Law}(\xi),\mathrm{Law}(\xi'))
\end{split}
\ee
holds uniformly on $\mathrm{Law}(\xi),\mathrm{Law}(\xi')\in\cP_1(\bbRD)$  and  $s\in[0,T].$
\item\label{item_3:contraction_coup_by_ref} The estimate
    \be\label{eq:contraction_coup_by_ref_2}
    \bbP[\bar{X}^x_s\neq \bar{X}^{x'}_s] \leq C^{-1}|x-x'| \frac{\lambda}{e^{\lambda s}-1}
    \ee
    holds for all $x,x'\in\bbRD$ and $s\in[0,T]$.
\een
\end{prop}

\begin{proof}
Items $(i)$ and $(ii)$ for time-independent drift fields $(\beta_s(\cdot)=\beta(\cdot))$ are proven at \cite[Thm. 1]{eberle2016reflection}. The adaptation of these arguments to the time-depenedent setting is straightforward and also implicitly done in \cite{eberle2019sticky}, where item $(iii)$ is obtained as a special case of Theorem 3 when the constant $M$ mentioned there is worth $0$. Earlier proofs of this result are to be found in \cite{lindvall1986coupling,chen1989coupling}.
\end{proof}
\subsection{Sticky coupling}

Sticky coupling of multidimensional diffusion processes as introduced in \cite{eberle2019sticky} is a coupling of the solution to \eqref{eq:SDE_1} with the solution of the same stochastic differential equation but relative to a different drift field $\tilde\beta$, i.e.
\be\label{eq:SDE_2}
\De \tilde{X}_s = \tilde\beta_s(\tilde{X}_s)\De s +\sigma\De B_s. 
\ee
Sticky coupling can be understood as the extension of coupling by reflection to the more general setting of different drift fields. The main difference with coupling by reflection lies in the fact the two coupled processes immediately move apart after meeting, instead of becoming one single process, which is clearly only possible if the two drift fields are the same. The construction of sticky coupling goes through an approximation procedure in which one alternates coupling by reflection, when the two processes are far apart, and synchronous coupling, when the processes get very close. We do not give the details here, but rather limit ourselves to state the results we are going to quote afterwards.

\begin{theorem}\label{thm:sticky_coupl}
Let $\beta,\tilde\beta:[0,T]\times\bbRD\longrightarrow\bbRD$ be locally Lipschitz continuous. Assume that 
\bes
\bar\kappa_{\beta}(r) \geq \kappa_0(r) \quad \forall r>0.
\ees
for some $\kappa_0:(0,+\infty)\rightarrow \bbR$ that is globally Lipschitz and equal to a positive constant outside a bounded interval. Moreover, assume that for each $s\in[0,T]$ there exist $M_s\in(0,+\infty)$ such that 
\bes
|\beta_s(x)-\tilde{\beta}_s(x)|\leq M_s, \quad \forall x\in\bbRD.
\ees
Then, for any $x,x'$ there exist pathwise unique strong solutions $(X^x_s)_{s\in[0,T]}$ and $(\tilde{X}^{x'}_s)_{s\in[0,T]}$ to \eqref{eq:SDE_1} and \eqref{eq:SDE_2} with initial condition $x$ and $x'$ respectively. Moreover, there exist a coupling $(\bar{X}^x_s,\bar{X}^{x'}_s)_{s\in[0,T]}$ of $(X^x_s)_{s\in[0,T]}$ and $(\tilde{X}^{x'}_s)_{s\in[0,T]},$ and a real-valued process $(r_s)_{s\in[0,T]}$ such that
\bes
\text{almost surely}, \quad |\bar{X}^x_s-\bar{X}^{x'}_s|\leq r_s \quad \forall s\in[0,T]
\ees
and $(r_s)_{s\in[0,T]}$ is a solution to the SDE
\bes
\De r_s =( M_s+\kappa_0(r_s)r_s)\De s +2\sigma  \mathbf{1}_{\{r_s>0\}}\,\De W_s, \quad r_0=|x-x'|,
\ees
where $(W_s)_{s\in[0,T]}$ is a one-dimensional Brownian motion.
\end{theorem}
\begin{proof}
This result is proven as a part of \cite[Thm. 3]{eberle2019sticky}, when $M_s$ is independent of $s$. Even in this case the adaptation to the time-dependent setting is straightforward.
\end{proof}
\section{Proofs}\label{sec:proofs}
We proceed to the proof of the main results that we break down in several Lemmas and Propositions. The full proof of Theorem \ref{thm:SP_val_fun} and Theorem \ref{thm:SP_turnpike_1} will then be easily obtained at the end of the section assembling together the intermediate results.

 \paragraph{Relaxed control problem}
 In the proof of Lemma \ref{lemma:SP_grad_est} it will be convenient to work with a relaxed version of problem \eqref{eq:SP_prob} that we now introduce following closely \cite{fleming2006controlled}. Given $0\leq t\leq T$, we call reference probability system a quadruple 
 \bes
 \nu=(\check\Omega,(\check{\cF}_s)_{s\in[t,T]},\check\bbP,(\check{B}_s)_{s\in[t,T]})
 \ees
 such that $(\check\Omega,(\check\cF_s)_{s\in[t,T]},\check\bbP)$ is a filtered probability space and $(\check{B}_s)_{s\in[t,T]}$ a $\check\cF_s$-adapted Brownian motion. We also call $\cU^{\nu}_{[t,T]}$ the set of admissible processes for $\nu$, that is to say, the set of all $\check\cF_s$-progressively measurable processes with moments of all order, see \eqref{eq:mom_cond}. Note that $(\check\cF_s)_{s\in[t,T]}$ needs not be the sigma algebra generated by $(\check{B}_{s})_{s\in[t,T]}$. The value function $\varphi^{T,g,\nu}$ relative to the reference probability system $\nu$ as well as the function $\varphi^{T,g,\mathrm{PM}}$, obtained optimizing over all reference probability systems are defined as follows:
 \bes
\varphi^{T,g,\nu}_t(x) = \inf_{u\in\cU^{\nu}_{[t,T]}} J_{t,x}^{T,g,\nu}(u), \qquad \varphi^{T,g,\mathrm{PM}}_t(x) = \inf_{\nu} \varphi^{T,g,\nu}_t(x). 
 \ees
In the above, $\mathrm{PM}$ stands for "progressively measurable" and optimization is performed over all reference probability systems. For any reference probability system $\nu$ and $u\in\cU_{[t,T]}^{\nu}$, $J_{t,x}^{T,g,\nu}$ is defined in the obvious way:

\be\label{eq:SP_obj_nu}
 J_{t,x}^{T,g,\nu}(u)=\bbE\left[\int_t^T F(X^{t,x,u,\nu}_s,u_s)\De s + g(X^{t,x,u,\nu}_T)\right],
\ee

where 
\begin{equation}\label{eq:contr_state_nu}
    \begin{cases}
         \De X^{t,x,u,\nu}_s= b(X^{t,x,u,\nu}_s,u_s)\,\De s + \sigma \De \check{B}_s,\\ 
         X^{t,x,u,\nu}_t=x. 
    \end{cases}      
\end{equation}
Existence of strong solutions and pathwise uniqueness for \eqref{eq:contr_state_nu} is proven under the current assumptions in \cite[Appendix D]{fleming2006controlled}. We shall see in the next proposition that the proposed relaxation is tight.
\subsection{Optimality conditions}\label{prop:SP_opt_cond}
\begin{prop}
Let Assumption \ref{ass:SP_wellposed}-\ref{ass:SP} hold. Moreover, assume that $g\in C^3_{\lip}(\bbRD)$.
\ben[(i), ref=\theprop(\roman*)]
\item\label{item_1:SP_opt_cond}  The value function $\varphi^{T,g}$ is a classical solution of the  Hamilton-Jacobi-Bellman equation \eqref{eq:SP_HJB} in $C^{1,2}_p([0,T)\times\bbRD)$ and uniqueness of classical solutions holds in the set \bes 
\{\varphi \in C^{1,2}_p([0,T)\times\bbRD):\sup_{s\in[0,T]} \|\varphi_s \|_{\lip}<+\infty \}.
\ees
\item\label{item_2:SP_opt_cond}  The functions $\varphi^{T,g}_t$ and $\varphi^{T,g,\mathrm{PM}}_t$ coincide for all $0\leq t\leq T$.
\item Recall the definition of $w(x,p)$ through
\bes
w(x,p)=\arg\min_{u\in\bbR^p}F(x,u)+b(x,u)\cdot p
\ees
and consider the drift field 
\bes
(s,x)\mapsto b(x,w(x,\nabla\varphi^{T,g}_s(x)))=-D_pH(x,\nabla\varphi^{T,g}_s(x)).
\ees
Then, for any $\xi$ with finite first moment the stochastic differential equation 
\be\label{eq:SP_opt_cond_3}
\begin{cases}
\De X_s= - D_pH(X_s, \nabla\varphi^{T,g}_s(X_s))\De s +\sigma\De B_s,\quad s\in[t,T]\\
X_t=\xi
\end{cases}
\ee
admits a strong solution $(\bmX^{t,\xi,T,g}_s)_{s\in[t,T]}$ and pathwise uniqueness holds. 
\item\label{item_4:SP_opt_cond} The map
\bes 
(s,x)\mapsto  w(x,\nabla\varphi^{T,g}_s(x))
\ees 
is an optimal Markov control policy and $(\bmX^{t,x,T,g}_s)_{s\in[t,T]}$ is an optimal process in the following sense: for any $x\in\bbRD$ $0\leq t\leq T$ 
\bes
\bmu^{t,x,T,g}_{\cdot}:=w(\bmX^{t,x,T,g}_{\cdot},\nabla\varphi^{T,g}_{\cdot}(\bmX^{t,x,T,g}_{\cdot})) \in\arg\min_{u \in \cU_{[t,T]}}  J_{t,x}^{T,g}(u).
\ees
and 
\bes
\bbP-\text{a.s.},\qquad\bmX^{t,x,T,g}_s= X^{t,x,\bmu^{t,x,T,g}}_s \quad \forall s\in[t,T].
\ees
\een
\end{prop}
The proof is a technical modification of classical argument and is therefore postponed to the Appendix. From now on, for $g\in C^3_\mathrm{Lip}(\bbRD)$ we denote the optimal drift by $\beta_s^{T,g}$, i.e.
\be\label{eq:contraction_SP_2}
\beta^{T,g}_s(x): = -D_pH (x,\nabla\varphi^{T,g}_s(x)))=b(x,w(x,\nabla\varphi^{T,g}_s(x)).
\ee
The optimality of solutions of \eqref{eq:SP_opt_cond_3} will be later established at Proposition \ref{prop:policy_opt} weakening the Assumption that $g\in C^3_{\lip}(\bbRD)$ into $g\in C^1_{\lip}(\bbRD)$.

\subsection{Uniform gradient estimate}
The following lemma is the backbone of our proof strategy. In there, we introduce a new variant of coupling by reflection adapted to the analysis of controlled diffusion processes. For this reason, we name it \emph{controlled reflection coupling}. Note that in this proof we cannot directly use Propositon \ref{prop:contraction_coup_by_ref} as it refers to uncontrolled SDEs. However, we can still profit from the properties of the the maps $\bmf,\bmC,\bm\lambda$ summarized at Propositon \ref{prop:Lyapunov_funct_coup_by_ref}. In the statement, we need the following twisted version of the Lipschitz norm, that we define for $g\in\mathrm{Lip}(\bbRD)$ and $f\in F$ as follows. 
\bes 
\| g \|_{f} = \sup_{\substack{x,x'\in\bbRD \\ x\neq x'}} \frac{|g(x)-g(x')|}{f(|x-x'|)}.
\ees

\begin{lemma}\label{lemma:SP_grad_est}
Let Assumption \ref{ass:SP_wellposed}-\ref{ass:SP} hold and set
\bes
(f,\lambda,C)=(\bmf(\bar\kappa_b),\bm\lambda(\bar\kappa_b),\bmC(\bar\kappa_b)),
\ees 
where $\bar\kappa_b$ is defined at \eqref{eq:eberle_kappa}. Then the estimate
\be\label{eq:SP_grad_est_6}
    \| \varphi^{T,g}_t\|_{f}\leq \Big(\,M^F_x\frac{1-e^{-\lambda(T-t)}}{C \lambda}+\|g\|_{f}\, e^{-\lambda(T-t)}\Big) \qquad  
\ee
holds uniformly on $0\leq t\leq T$ and $g\in\mathrm{Lip}(\bbRD)$. In particular, 
\be\label{eq:SP_grad_est_3}
   \|\varphi^{T,g}_t\|_{\lip}\leq C^{-1}\Big(M^F_x\frac{1-e^{-\lambda(T-t)}}{\lambda}+M_x^g e^{-\lambda(T-t)}\Big):=M^{\varphi,g}_{x,T-t}
\ee
holds uniformly on $0\leq t\leq T$ and 
\be\label{eq:SP_grad_est_7}
\sup_{ 0\leq t\leq T} \| \varphi^{T,g}_t(x)\|_{\lip}\leq  C^{-1}\Big(\frac{M^F_x}{\lambda}+M_x^g \Big): =M_x^{\varphi,g}.
\ee

\end{lemma}

\begin{proof}
We first carry out the proof assuming $g\in C^3_{\lip}(\bbRD)$ and eventually relax this hypothesis. Let $x,x'\in\bbRD,0\leq t\leq T$ be fixed and $(\bmX^{t,x,T,g}_s)_{s\in[t,T]}$ be the optimal process given by \eqref{eq:SP_opt_cond_3} for the initial condition $\xi\equiv x$. In the rest of the proof, since there is no ambiguity, we abbreviate $\bmX^{t,x,T,g}_s$ with $\bmX_s$. Following a similar convention, we denote by $\bmu$ the optimal control $\bmu^{t,x,T,g}=w(\bmX_s,\nabla\varphi^{T,g}_s(\bmX_s))$.
Next, we consider the diffusion process $(\bmX_s,X'_s)_{s\in[t,T]}$ on $\bbR^{2d}$ defined by $(\bmX_t,X'_t)=(x,x')$ and 
\be\label{eq:SP_grad_est_5}
\begin{cases}
    \De \bmX_s= b(\bmX_s,\bmu_s)\De s + \sigma \De B_s, \quad &\text{for}\, t\leq s<\tau, \\
    \De X'_s= b(X'_s,\bmu_s)\De s + \sigma \De \check{B}_s, \quad \quad &\text{for}\, t\leq s\leq \tau,\,\, X'_s=\bmX_s \,\, \text{for}\,\, \tau\leq s\leq T,
\end{cases}
\ee
where 
\be\label{eq:SP_grad_est_4}
\De\check{B}_s := (\mathrm{I}-2\rme_s\cdot\rme^{\top}_s\mathbf{1}_{\{s<\tau\}})\cdot\De B_s, \quad \rme_s:= \frac{\bmX_s-X'_s}{|\bmX_s-X'_s|},
\ee
and

\bes
    \tau =\inf\{ s\geq t: \bmX_s=X'_s\}\wedge T.
\ees
Under the current assumptions on $b(\cdot,\cdot)$, and because of Proposition \ref{item_1:SP_opt_cond}, the process $(\bmX_s,X'_s)_{s\in[t,T]}$ can be realized as the unique strong solution of the SDE \eqref{eq:SP_grad_est_5}. Moreover, by L\'evy's characterization, $(\check{B}_s)_{s\in[t,T]}$ is an $(\cF_s)_{s\in[t,T]}$-Brownian motion. But then, $\nu=(\Omega,(\cF_s)_{s\in[t,T]},\bbP,(\check{B}_s)_{s\in[t,T]})$ is a reference probability system and, because of Proposition \ref{item_4:SP_opt_cond}  $\bmu$ is an admissible control, i.e. $\bmu\in\cU^{\nu}_{[t,T]}$\footnote{ Note that by construction $\cU^{\nu}_{[t,T]}=\cU_{[t,T]}$}. Recalling \eqref{eq:contr_state_nu} we find
\be\label{eq:SP_grad_est_9}
\bbP-\text{a.s.\,,} \quad X'_{s}=X^{t,x',\bmu,\nu}_{s} \quad \forall s\in [t,T].
\ee
whence, thanks to Proposition \ref{item_2:SP_opt_cond},
\be\label{eq:SP_grad_est_2}
\varphi^{T,g}_t(x')=\varphi^{T,g,\mathrm{PM}}_t(x')\leq \varphi^{T,g,\nu}_t(x')\stackrel{\eqref{eq:SP_grad_est_9}}{\leq}  \bbE\left[\int_{t}^TF(X'_s,\bmu_s)\De s + g(X'_T)\right].
\ee
By hypothesis $\bar\kappa_b\in K$, where $K$ is defined at \eqref{eq:Lyapunov_funct_coup_by_ref_1}. Thus, we can consider
\[f:=\bmf(\bar\kappa_b),\quad\lambda:=\bm\lambda(\bar\kappa_b)\quad \text{and} \quad C:=\bmC(\bar\kappa_b)\]
as per Definition \ref{def:Lyapunov_funct_coup_by_ref}. Next,  define the process $r_s=|\bmX_s-X'_s|$. By construction we have that for $s<\tau$
\bes
\De (\bmX_s-X'_s) = (b(\bmX_s,\bmu_s)- b(X'_s,\bmu_s))\De s +2 \sigma\,\rme_s \De W_s,
\ees
where $(W_s)_{s\in[t,T]}$ is the one-dimensional Brownian motion 
\bes
     \De W_s =  \rme_s^{\top} \cdot \De B_s.
\ees
An application of It\^o's formula gives that for $s<\tau$
\bes
\begin{split}
\De r_s&= \frac{1}{r_s}\ip{ b(\bmX_s,\bmu_s)- b(X'_s,\bmu_s)}{\bmX_s-X'_s}\,\De s + 2\sigma\De W_s\\
&+\underbrace{ \frac{2\sigma^2}{r_s} \rme_s^{\top}\big( \mathrm{I}-\rme_s\rme^{\top}_s\big)\rme_s\De s}_{=0}\\
&\stackrel{\eqref{eq:SP_drift_ass_1}}{\leq} -\frac{\sigma^2}{2}\bar\kappa_b(r_s) r_s\De s +2 \sigma \De W_s.
\end{split}
\ees
With another application of It\^o formula we obtain
\bes
    \begin{split}
        \De f(r_s) &\stackrel{f\in F}{\leq} 2\sigma^2\Big(-f'(r_s)\frac{\bar\kappa_b(r_s)}{4}r_s  +f''(r_s)\Big) \De s +2\sigma f'(r_s) \De W_s\\
        &\stackrel{\eqref{eq:Eberle_funct_2}}{\leq} -\lambda f(r_s)\De s + 2\sigma f'(r_s) \De W_s.
    \end{split}
\ees
Rewriting the above in integral form, observing that $r_s\equiv 0$ for $s\geq\tau$, taking expectation on both sides and eventually using Gronwall's Lemma we find
\be\label{eq:SP_grad_est_8}
\bbE[f(|\bmX_s-X'_s|)] = \bbE[f(r_s)] \leq f(|x-x'|)e^{-\lambda(s-t)} \quad \forall t\leq s\leq T.
\ee
We can now conclude using \eqref{eq:SP_grad_est_2} and the optimality of $\bmX_s$. Indeed we have
\bes
\begin{split}
\varphi^{T,g}_t(x')-\varphi^{T,g}_t(x) & \stackrel{\eqref{eq:SP_grad_est_9}}{\leq}  \bbE[\int_{t}^TF(X'_s,\bmu_s)-F(\bmX_s,\bmu_s)\De s +g(X'_T)-g(\bmX_T)]\\
&\stackrel{\eqref{eq:Lyapunov_funct_coup_by_ref_2}}{\leq} \bbE\left[C^{-1} M^F_x\int_t^T f(|X'_s-\bmX_s|)\De s + \|g\|_f f(|X'_T-\bmX_T|)\right]\\
&\stackrel{\eqref{eq:SP_grad_est_8}}{\leq} f(|x-x'|)\Big( M^F_x \,\frac{1-e^{-\lambda(T-t)}}{\lambda C }
+\|g\|_f\,e^{-\lambda (T-t)}\Big) 
\end{split}
\ees
which gives \eqref{eq:SP_grad_est_6} since $x,x'$ can be chosen arbitrarily. The other two bounds \eqref{eq:SP_grad_est_3}\eqref{eq:SP_grad_est_7} are then simply deduced using \eqref{eq:Lyapunov_funct_coup_by_ref_2}. The last thing that remains to be done is to remove the assumption that $g\in C^3(\bbRD)$. To do so, observe that for any $g\in \lip(\bbRD)$, one can construct a sequence $(g^M)_{M\geq 1}\subseteq C^{3}_{\lip}(\bbRD)$ such that 
\be\label{eq:SP_grad_est_10}
\lim_{M\rightarrow+\infty}\sup_{x\in\bbRD}|g^M-g |=0, \quad \lim_{M\rightarrow+\infty} \|g^M\|_{f}=\|g\|_{f},
\ee
 by means of a standard mollification procedure. From this, it follows that for any $0\leq t \leq T$ we have
\bes
\lim_{M\rightarrow+\infty}\sup_{x\in\bbRD}|\varphi^{T,g^M}_t(x)-\varphi^{T,g}_t(x) |=0,
\ees
which in particular imply
\be\label{eq:SP_grad_est_11}
\|\varphi^{T,g}_t\|_{f}\leq\liminf_{M\rightarrow+\infty}\|\varphi^{T,g^M}_t \|_{f}.
\ee
At this point, the proof of \eqref{eq:SP_grad_est_6} is obtained passing to the limit in 
\bes
    \| \varphi^{T,g^M}_t\|_{f}\leq \Big(\,M^F_x\frac{1-e^{-\lambda(T-t)}}{C \lambda}+\|g^M\|_{f}\, e^{-\lambda(T-t)}\Big) \qquad 
\ees
and using \eqref{eq:SP_grad_est_10},\eqref{eq:SP_grad_est_11}. The proof \eqref{eq:SP_grad_est_3}\eqref{eq:SP_grad_est_7} follows from \eqref{eq:Lyapunov_funct_coup_by_ref_2}.
\end{proof}
In the upcoming auxiliary lemma we draw some algebraic consequences from the gradient estimate we have just obtained. 
\begin{lemma}\label{lem:coeff_bound}
Let Assumption \ref{ass:SP_wellposed}-\ref{ass:SP} hold. Moreover, assume that $\varphi^{T,g}_t\in C^1_{\lip}(\bbRD)$ and that \eqref{eq:SP_grad_est_7} holds. Then we have

\ben[(i)]
\item The estimates 
\begin{subequations}
\be\label{eq:control_bound_1}
|w(x,\nabla\varphi^{T,g}_s(x))| \leq \frac{M_u(1+M_x^{\varphi,g})}{\omega_{M_x^{\varphi,g}}},
\ee
\be\label{eq:drift_lipschitz_1}
|b(x,w(x,\nabla\varphi^{T,g}_s(x)))-b(x,0)|\leq \frac{M_u^2(1+M_x^{\varphi,g})}{\omega_{M_x^{\varphi,g}}},
\ee
\be\label{eq:Hess_Ham_bound}
|D_{pp}H(x,p)| \leq\frac{ M_u^2}{\omega_{|p|}} ,
\ee
\end{subequations}
hold uniformly on $x,p\in\bbRD,0\leq s\leq T$.
\item The estimates
\begin{subequations}
\be\label{eq:coeff_bound_4}
D_xH(x,\nabla\varphi^{T,g}_s(x)) \leq M_xM_x^{\varphi,g}+M^F_x:=M_x^{H,g}
\ee
\be\label{eq:coeff_bound_2} |D_{xp}H|(x,\nabla\varphi^{T,g}_s(x))\leq \frac{ M_{xu}M_u(1+M_x^{\varphi,g})}{\omega_{M_x^{\varphi,g}}}+M_x:=M^{H,g}_{xp}
\ee
\be\label{eq:coeff_bound_3}
|D_{xx} H|(x,\nabla\varphi^{T,g}_s(x))\leq M_{xx}(1+M_x^{\varphi,g})+\frac{M^2_{xu}(1+M_x^{\varphi,g})^2}{\omega_{M_x^{\varphi,g}}}:=M^{H,g}_{xx}
\ee
\end{subequations}
hold uniformly on $x\in\bbRD$ and $0\leq s\leq T$.
\item The estimates
\begin{subequations}
\be\label{eq:coeff_bound_11}
 |D_xw|(x,p) \leq\frac{M_{xu}(1+|p|)}{\omega_{|p|}}
\ee
\be\label{eq:coeff_bound_13}
|D_pw|(x,p) \leq \frac{M_u}{\omega_{|p|}}
\ee
\end{subequations}
hold uniformly on $x,p\in\bbRD$.
\een
\end{lemma}
The proof of Lemma \ref{lem:coeff_bound} is deferred to the Appendix section. In the next Theorem we prove exponential stability with respect to the initial condition. 
\begin{theorem}\label{thm:contraction_SP}
Let Assumption \ref{ass:SP_wellposed}-\ref{ass:SP} hold. Moreover, assume that $g\in C^3_{\lip}(\bbRD)$. 
Then for $t\leq s\leq T$ we have 
\be\label{eq:contraction_SP_3}
\kappa_{\beta^{T,g}_s}(r) \geq \kappa(r) \quad \forall r>0, \quad \text{with} \quad \kappa(r)=\bar\kappa_b(r) -\frac{4M_u^2(1+M_x^{\varphi,g})}{\sigma^2\omega_{M_x^{\varphi,g}}\,\,r}.
\ee
Moreover, setting 
\be\label{eq:contraction_SP_5}
f=\bmf(\kappa),\overrightarrow{\lambda}=\bm\lambda(\kappa),\overrightarrow{C}=\bmC^{-1}(\kappa)\ee
the estimates
\begin{subequations}
\be\label{eq:contraction_SP_1_a}
W_f(\mathrm{Law}(\bmX^{t,\xi,T,g}_s),\mathrm{Law}(\bmX^{t,\xi',T,g}_s)) \leq  W_f(\mathrm{Law}(\xi),\mathrm{Law}(\xi'))\exp(-{\overrightarrow{\lambda}} (s-t))
\ee
\be\label{eq:contraction_SP_1_b}
W_1(\mathrm{Law}(\bmX^{t,\xi,T,g}_s),\mathrm{Law}(\bmX^{t,\xi',T,g}_s)) \leq{\overrightarrow{C}} W_1(\mathrm{Law}(\xi),\mathrm{Law}(\xi'))\exp(-{\overrightarrow{\lambda}} (s-t))
\ee
\end{subequations}

hold uniformly on $0\leq t\leq s\leq T$ and $\mathrm{Law}(\xi),\mathrm{Law}(\xi')\in\cP_1(\bbRD)$.

\end{theorem}

\begin{proof}
From Proposition \ref{prop:SP_opt_cond} we know that $(\bmX^{t,\xi,T,g}_s)_{s\in[t,T]}$ and $(\bmX^{t,\xi',T,g}_s)_{s\in[t,T]}$ are two strong solutions of the SDE \eqref{eq:SDE_1} with
$\beta_s=\beta_s^{T,g}$ and initial conditions $\xi$ and $\xi'$ respectively. We have for all $x,x'\in\bbRD$ and $t\leq s\leq T$:
\bes
\begin{split}
   &-\ip{\beta_s^{T,g}(x)-\beta_s^{T,g}(x')}{x-x'}\\
    &=-\ip{\beta_s^{T,g}(x)-b(x,0)}{x-x'}+\ip{\beta_s^{T,g}(x')-b(x',0)}{x-x'}\\
    &-\ip{b(x,0)-b(x',0)}{x-x'}.
\end{split}
\ees
We have
\bes
\begin{split}
|\ip{\beta_s^{T,g}(x)-b(x,0)}{x-x'}|=|\ip{b(x,w(x,\nabla\varphi^{T,g}_s(x)))-b(x,0)}{x-x'}|\\ \stackrel{\eqref{eq:drift_lipschitz_1}}{\leq} \frac{M_u^2(1+M_x^{\varphi,g})}{\omega_{M_x^{\varphi,g}}} |x-x'|
\end{split}
\ees
and the term $|\ip{\beta_s^{T,g}(x')-b(x',0)}{x-x'}|$ can be bounded in the same way. But then, we deduce that 
\bes
\kappa_{\beta^{T,g}_s}(r) \geq \bar\kappa_b(r) -\frac{4M_u^2(1+M_x^{\varphi,g})}{\sigma^2\omega_{M_x^{\varphi,g}}\,r} \quad \forall r\geq 0,t\leq s\leq T,
\ees
which is \eqref{eq:contraction_SP_3}. Moreover, since $\bar\kappa_b\in K$ by assumption, we conclude that $\kappa\in K$, where $\kappa$ is defined at \eqref{eq:contraction_SP_3}. The desired conclusion then follows from  Proposition \ref{item_2:contraction_coup_by_ref}.
\end{proof}

\subsection{Hessian bounds}
Here is the main result of this subsection. 
\begin{prop}\label{prop:SP_Hess_est}
Let Assumption \ref{ass:SP_wellposed}-\ref{ass:SP} hold and $g\in C^3_{\lip}(\bbRD)$ with  \[\sup_{x\in\bbRD}|\nabla^2 g |<+\infty.\] Moreover, let $\kappa$ be given by \eqref{eq:contraction_SP_3} and $(\lambda,C)=(\bm\lambda(\kappa),\bmC(\kappa))$. Then for all $0\leq t\leq T$ we have
\be\label{eq:SP_Hess_est_2}
\sup_{x\in\bbRD  } |\nabla^2\varphi^{T,g}_t(x)| \leq \inf_{\theta\in(0,T-t) }A_{\theta,T-t}\,\,e^{ M^{H,g}_{xp}\,\,\theta}:=M^{\varphi,g}_{xx,T-t},
\ee
where 
\bes
A_{\theta,T-t}:=C^{-1} \Big(2M_x^g \frac{\lambda}{e^{\lambda (T-t)}-1} + 2M^{H,g}_{x}\int_{\theta}^{+\infty} \frac{\lambda}{e^{\lambda s}-1} \De s+ \frac{M^{H,g}_{xx}}{\lambda}\Big).
\ees
\end{prop}
Note that the above constants do not depend in $\nabla^2 g$ and the boundedness assumption on $\nabla^2g$ will be later removed at Proposition \ref{cor:reg_eff}.
The proof of Proposition \ref{prop:SP_Hess_est} appeals to a representation of $\nabla\varphi^{T,g}_{t}(x)$ through the stochastic maximum principle, see \eqref{eq:SP_Hess_est_1} below. Under slightly different assumptions than ours, such representation is well known, see e.g. \cite[Thm 3.2, Ch.3]{yong1999stochastic}.
To prepare for the proof of the Hessian bound, we first establish some preliminary rough bounds on the value function and its derivatives.
\begin{prop}\label{prop:aux_hess_bd}
Let Assumption \ref{ass:SP_wellposed}-\ref{ass:SP} hold and $g\in C^3_{\lip}(\bbRD)$ with \[\sup_{x\in\bbRD}|\nabla^2 g (x)|<+\infty.\] Then, for any $T>0$ there exist $C_T\in(0,+\infty)$ such that 
\be\label{eq:aux_hess_bd_1}
\begin{split}
|H(x,\nabla\varphi^{T,g}(x))|+|\varphi^{T,g}_t(x)|+|\partial_t\varphi^{T,g}_t(x)|+|\nabla^2\varphi^{T,g}_t(x)|&\leq C_T(1+|x|),\\
|D_x H|(x,\nabla\varphi^{T,g}_t(x))+|\nabla\varphi^{T,g}_t(x)|&\leq C_T
\end{split}
\ee
hold uniformly on $0\leq t\leq T$ and $x\in\bbRD$.
\end{prop}
The proof of this technical result is postponed to the appendix.

\begin{proof}[Proof of Proposition \ref{prop:SP_Hess_est}]
Let $t,T>0$ and $x,x'\in\bbRD$. Since $(s,x)\mapsto \beta^{T,g}_s(x)$ is locally Lipschitz by Proposition \ref{item_1:SP_opt_cond} and $\bar\kappa_{\beta^{T,g}}\in K$ by Theorem \ref{thm:contraction_SP}, we can apply Proposition \ref{item_1:contraction_coup_by_ref} to conclude that coupling by reflection of $(\bmX^{t,x,T,g}_s)_{s\in[t,T]}$ and $(\bmX^{t,x',T,g}_s)_{s\in[t,T]}$ can be constructed as a strong solution to \eqref{eq:coup_by_ref} and we denote it $(\bar{\bmX}^x_s,\bar{\bmX}^{x'}_s)_{s\in[t,T]}$. We now consider a spatial regularization $\varphi^{T,g}$, i.e. we define
\be\label{eq:SP_Hess_est_11}
\varphi^{T,g,\varepsilon}=  \varphi^{T,g}\ast \gamma_{\varepsilon}, \quad \gamma_{\varepsilon}(y)= \frac{1}{\sqrt{2\pi\varepsilon^d}}\exp(-|y|^2/2\varepsilon).
\ee
Note that, thanks to the bounds \eqref{eq:aux_hess_bd_1} we can exchange derivatives and integrals in the $\varphi^{T,g}\ast\gamma^{\varepsilon}$. In particular, we have the following relations for all $1\leq i,j\leq d$:
\bes
\begin{split}
\partial_{x_i}\varphi^{T,g,\varepsilon}=(\partial_{x_i}\varphi^{T,g})\ast\gamma_{\varepsilon}, \quad \partial_{x_i x_j}\varphi^{T,g,\varepsilon}=(\partial_{x_ix_j}\varphi^{T,g})\ast\gamma_{\varepsilon},\\\
\partial_{x_j x_j}(\partial_{x_i}\varphi^{T,g,\varepsilon})=\partial_{x_i}\big((\partial_{x_j,x_j}\varphi^{T,g})\ast\gamma_{\varepsilon}\big),\quad \partial_{t}\big(\partial_{x_i}\varphi^{T,g,\varepsilon}\big)=\partial_{x_i}\big((\partial_{t}\varphi^{T,g})\ast \gamma_{\varepsilon}\big).
\end{split}
\ees
Using these identities and applying It\^o formula to the function $\nabla\varphi^{T,g,\varepsilon}$  we find that for any $\varepsilon>0$ and $y=x,x',s\leq T$
\bes
\begin{split}
(\nabla\varphi^{T,g}_s\ast\gamma_{\varepsilon})(\bar\bmX^y_s) &\stackrel{\eqref{eq:SP_HJB}}{=} (\nabla g\ast\gamma_{\varepsilon})(\bar\bmX^{y}_T )-\int_s^T D_x\Big(H(\cdot,\nabla\varphi^{T,g}_r(\cdot)) \ast \gamma_{\varepsilon}\Big)\big(\bar\bmX^{y}_r\big)\De r\\
&-\int_s^T \big(\nabla^2\varphi^{T,g}_r\cdot D_pH(\cdot,\nabla\varphi^{T,g}_r(\cdot))\big) \ast \gamma_{\varepsilon}\big(\bar\bmX^{y}_r\big)\De r\\
&+\int_s^T\nabla^2\varphi^{T,g,\varepsilon}_r\cdot D_pH(\bar\bmX^{y}_r)\De r +M_T-M_s\\
&=\nabla (g\ast\gamma_{\varepsilon})(\bar\bmX^{y}_T )-\int_s^T \big(D_xH(\cdot,\nabla\varphi^{T,g}_r(\cdot))\big) \ast \gamma_{\varepsilon}\big(\bar\bmX^{y}_r\big)\De r\\
&-\int_s^T \big(\nabla^2\varphi^{T,g}_r\cdot D_pH(\cdot,\nabla\varphi^{T,g}_r(\cdot))\big) \ast \gamma_{\varepsilon}\big(\bar\bmX^{y}_r\big)\De r\\
&+\int_s^T\nabla^2\varphi^{T,g,\varepsilon}_r\cdot D_pH(\bar\bmX^{y}_r)\De r+M_T-M_s,
\end{split}
\ees
where $M_T-M_s$ is a square integrable martingale.
Letting $\varepsilon\rightarrow 0$ and relying once again on \eqref{eq:aux_hess_bd_1} to justify the exchange of limits and integrals, we arrive at
\be\label{eq:SP_Hess_est_1}
\nabla\varphi^{T,g}_s(\bar\bmX^y_s)=\nabla g(\bar\bmX^{y}_T )-\int_s^T D_xH(\bar\bmX^{y}_r,\nabla\varphi^{T,g}_r(\bar\bmX^{y}_r))\De r + M_T-M_s,
\ee
where the above equality has to be understood in the almost sure sense. Fix now $\theta\in[0,T-t]$ and let $\kappa$ be as in \eqref{eq:contraction_SP_3}. Since $\bar\kappa_{\beta^{T,g}}\geq \kappa$ we can invoke Proposition \ref{item_3:contraction_coup_by_ref} to obtain
\be\label{eq:SP_Hess_est_5}
\begin{split}
\bbE\Big[\int_{t+\theta}^T| D_xH(\bar\bmX^{x}_s,\nabla\varphi^{T,g}_s(\bar\bmX^{x}_s))&-D_xH(\bar\bmX^{x'}_s,\nabla\varphi^{T,g}_s(\bar\bmX^{x'}_s))|\De s +|\nabla g(\bar\bmX^{x'}_T )-\nabla g(\bar\bmX^{x}_T )|\Big]\\
&\stackrel{\eqref{eq:coeff_bound_4}}\leq   2M_x^{H,g}\int_{t+\theta}^T \bbP[\bar{\bmX}^x_r\neq\bar{\bmX}^{x'}_r] \De r +2M_x^g \bbP[\bar{\bmX}^x_T\neq\bar{\bmX}^{x'}_T]\\
&\stackrel{\text{Prop.}\ref{item_3:contraction_coup_by_ref}}{\leq}2|x-x'|M^{H,g}_{x}\int_{\theta}^{T-t} \frac{\lambda}{C(\exp(\lambda r)-1)} \De r \\
&+2|x-x'| M_x^g \frac{C^{-1}\lambda}{\exp(\lambda(T-t))-1}:= |x-x'|\tilde{A}_{\theta,T-t}.
\end{split}
\ee
Next, define for $s\in[t,t+\theta]$ the function $h(s)=\bbE[|\nabla\varphi^{T,g}_s(\bar\bmX^{x}_s)-\nabla\varphi^{T,g}_s(\bar\bmX^{x'}_s)|]$. From \eqref{eq:SP_Hess_est_1},\eqref{eq:SP_Hess_est_5} and Proposition \ref{prop:contraction_coup_by_ref} we directly obtain
\bes
\begin{split}
h(s) &\stackrel{\substack{\eqref{eq:coeff_bound_2}\\\eqref{eq:coeff_bound_3}}}{\leq} \int_s^{t+\theta} M^{H,g}_{xx}\bbE[|\bar\bmX^{x}_r-\bar\bmX^{x'}_r|] +M^{H,g}_{xp}\bbE[|\nabla\varphi^{T,g}_r(\bar\bmX^{x}_r)-\nabla\varphi^{T,g}_r(\bar\bmX^{x'}_r)|]\De r+|x-x'|\tilde{A}_{\theta,T-t}\\
     &\stackrel{\text{Prop.}\, \ref{item_2:contraction_coup_by_ref}}{\leq} \int_s^{t+\theta} M^{H,g}_{xx}|x-x'|C^{-1}\exp(-\lambda(r-t)) +M^{H,g}_{xp}h(r)\De r+|x-x'|\tilde{A}_{\theta,T-t}\\
 &\leq \Big(\frac{M^{H,g}_{xx}}{\lambda C}+\tilde{A}_{\theta,T-t}\Big)|x-x'|+ M^{H,g}_{xp}\int_s^{t+\theta}h(r)\De r\\
 &=|x-x'|A_{\theta,T-t}+ M^{H,g}_{xp}\int_s^{t+\theta}h(r)\De r.
 \end{split}
\ees
Defining $\tilde{h}(s')=h(t+\theta-s')$ for $s'\in[0,\theta]$ we can rewrite the above as
\bes
\tilde{h}(s') \leq  |x-x'| A_{\theta,T-t} +  M^{H,g}_{xp}\int_0^{s'} \tilde{h}(r')\De r' \quad \forall s'\in[0,\theta].
\ees
An application of Gronwall's lemma gives 
\bes
 \tilde{h}(\theta)= |\nabla\varphi^{T,g}_t(x)-\nabla\varphi^{T,g}_t(x')|\leq |x-x'| A_{\theta,T-t}e^{M^{H,g}_{xp}\theta}
\ees
The desired conclusion follows letting $|x-x'|\rightarrow0$.
\end{proof}
There are a number of interesting consequences that can be drawn from the Hessian estimates.
\begin{prop}\label{cor:reg_eff}
Let Assumption \ref{ass:SP_wellposed}-\ref{ass:SP} hold.
 
 \ben[(i)]
 \item If $g\in\mathrm{Lip}(\bbRD)$, $\varphi^{T,g}_t\in C^1_{\lip}(\bbRD)$ and $\|\nabla\varphi^{T,g}_t\|_{\mathrm{Lip}}\leq M^{\varphi,g}_{xx,T-t}$ for all $0\leq t< T$.
\item If $g\in\lip(\bbRD)$, $[0,T]\times\bbRD\ni(t,x)\mapsto \varphi^{T,g}_t(x)$ is a viscosity solution to \eqref{eq:SP_HJB}.
\item If \label{item_3:reg_eff} $g\in C^1_\lip(\bbRD)$, then the SDE \eqref{eq:Ham_SDE} admits a strong solution $(\bmX^{t,\xi,T,g}_s)_{s\in[0,T]}$ for any $\xi$ with $\mathrm{Law}(\xi)\in\cP_1(\bbRD)$ and pathwise uniqueness holds.
\een
\end{prop}

\begin{proof}
Fix $g\in\mathrm{Lip}(\bbRD)$ and $0\leq t <T$. Then there exist a sequence $(g^M)_{M\in\N}\subseteq C^3_{\lip}(\bbRD)$ such that $\nabla^2g^M$ is bounded for all $M$ and
\bes 
\lim_{M\rightarrow +\infty} \sup_{x\in\bbRD}|g^M-g|(x)=0,\quad \lim_{M\rightarrow+\infty} \|g^M \|_{\mathrm{Lip}}= \|g\|_{\mathrm{Lip}}
\ees
From these properties, it follows that 
\be\label{eq:reg_eff_1}
\lim_{M\rightarrow+\infty}\sup_{\substack{t\in[0,T]\\x\in\bbRD}}|\varphi^{T,g^M}_t(x) -\varphi^{T,g}_t(x) |=0.
\ee
Since the constants $M_x^{\varphi,g^M},M^{\varphi,g^M}_{xx}$ depend on $g^M$ only through $\|g^M\|_{\lip}$, applying the gradient estimate \eqref{eq:SP_grad_est_7} and the Hessian estimate \eqref{eq:SP_Hess_est_2} we conclude that for all $\varepsilon$ 
\bes
\sup_{M\in\N} \sup_{\substack{0\leq t\leq T-\varepsilon \\ x\in\bbRD}}| \nabla \varphi^{T,g^M}_t(x)|+ | \nabla^2\varphi^{T,g^M}_t(x)|<+\infty.
\ees
But then, by Arzéla-Ascoli Theorem we find that for any $t<T$ there exist $\Phi\in C_{\lip}(\bbRD;\bbRD)$ such that,  along a non-relabeled subsequence we have that for any compact set $K$
\be\label{eq:reg_eff_2}
\lim_{M\rightarrow+\infty}\sup_{x\in K}|\nabla\varphi^{T,g^M}_t(x)-\Phi(x)|=0.
\ee
But then, because of \eqref{eq:reg_eff_1}, we can conclude that $\varphi^{T,g}_t\in C^1(\bbRD)$ and that $\nabla\varphi^{T,g}_t=\Phi$. In particular, $\nabla\varphi^{T,g}_t\in\lip(\bbRD;\bbRD)$ and, in view of \eqref{eq:SP_Hess_est_2}, we also have $\| \nabla\varphi^{T,g}_t\|_{\lip}\leq M_{xx,T-t}^{\varphi,g}$. The fact that $\varphi^{T,g}_s$ is a viscosity solution is a direct consequence of the fact that $\varphi^{T,g^M}$ is a viscosity solution for all $M$ since it is a classical solution, and of a well known stability property for viscosity solution under uniform convergence, see \cite[Lemma 6.2]{fleming2006controlled}. Fix now $x\in\bbRD$. Then, existence and uniqueness of a strong solution $(\bmX^{0,x,T,g}_s)_{s\in[0,T)}$ for \eqref{eq:Ham_SDE} on $[0,T)$ with initial condition $\xi\equiv x$ is easily deduced from classical results (see e.g. \cite[Thm 5.2.1]{oksendal2013stochastic}) since for any $\varepsilon>0$, the restriction of $\beta^{T,g}$ to $[0,T-\varepsilon)\times\bbRD$ grows at most linearly and is uniformly Lipschitz in the space variable. Indeed thanks to Proposition \ref{prop:SP_Hess_est} and \eqref{eq:Hess_Ham_bound},\eqref{eq:coeff_bound_2} and what we have just shown, we have
\be\label{eq:reg_eff_3}
\forall (s,x)\in[0,T-\varepsilon]\times\bbRD,\quad |\beta^{T,g}_s(x)-\beta^{T,g}_s(x')| \leq \Big(M_{xp}^{H,g}+ \frac{M^2_u}{ \omega_{M^{\varphi,g}_x}} M^{\varphi,g}_{xx,\varepsilon}\Big) |x-x'|.
\ee Next we observe that, repeating the same argument used in the proof of Theorem \ref{thm:contraction_SP}, we find that \eqref{eq:contraction_SP_3} holds true. But then, a standard calculation using Gronwall Lemma gives that 
\bes
\bbP- \text{a.s},\quad \sup_{s\in[0,T)}|\beta^{T,g}_s(\bmX^{0,x,T,g}_s)|<+\infty.
\ees
From this, it follows that
\bes
\bbP- \text{a.s},\quad \lim_{h\rightarrow 0}\int_0^{T-h}\beta^{T,g}_s(\bmX^{0,x,T,g}_s)\De s \,\,\text{exists},
\ees
from which we obtain that $\bmX^{0,x,T,g}_T$ is well-defined as the almost sure limit of $\bmX^{0,x,T,g}_{T-h}$ as $h\rightarrow0$ and that the process $(\bmX^{0,x,T,g}_s)_{s\in[0,T]}$ is a strong solution for \eqref{eq:Ham_SDE}.  Pathwise uniqueness of strong solutions is obtained in a standard way leveraging the global Lipschitzianity of $\beta^{T,g}$ on $[0,T-\varepsilon)\times\bbRD$ for all $\varepsilon>0$. We have therefore proven item (iii) under the additional assumption that the initial condition is deterministic. The extension to the general case is standard.
  \end{proof}
\subsection{Exponential stability with respect to the final condition}\label{sec:SP_stab_fin_cond}

\begin{lemma}\label{lemma:SP_stab_fin_cond}
   Let Assumption \ref{ass:SP_wellposed}-\ref{ass:SP} hold. Moreover, assume that $g,g'\in C^1_{\lip}(\bbRD)$. Then, if we define
    \be\label{eq:SP_stab_fin_cond_1}
    \begin{split}
    \kappa(r) &=\bar\kappa_b(r)-\frac{M_u^2}{\sigma^2 r}\Big(\frac{4(1+M_x^{\varphi,g})}{\omega_{M_x^{\varphi,g}}}+\frac{M^{\varphi,g}_x+M^{\varphi,g'}_x}{ \omega_{M^{\varphi,g}_x \vee M^{\varphi,g'}_x}} \,\Big)\quad \forall r>0, \\
     (f,\lambda,C)&=(\bmf(\kappa),\bm\lambda(\kappa),\bmC(\kappa)).
     \end{split}
    \ee
   then the estimate
    \begin{equation}\label{eq:SP_stab_fin_cond_4}
     \|\varphi^{T,g}_t-\varphi^{T,g'}_t\|_{f} \leq \|g-g'\|_{f}  \exp(-\lambda (T-t))
    \end{equation}
    holds uniformly on  $0\leq t\leq T$. In particular,
    \begin{equation}\label{eq:SP_stab_fin_cond_2}
     \|\varphi^{T,g}_t-\varphi^{T,g'}_t\|_{\mathrm{Lip}} \leq \|g-g'\|_{\mathrm{Lip}} C^{-1} \exp(-\lambda (T-t))
    \end{equation}
    holds uniformly on  $0\leq t\leq T$.
    
\end{lemma}

\begin{proof}
We first assume that $g,g'\in C^3_{\lip}(\bbRD)$ and define $\psi=\varphi^{T,g'}-\varphi^{T,g}$. From Taylor's formula we know that for any $(s,x)\in[t,T)\times\bbRD$ there exists 
\bes v_s(x)\in \{ \theta \nabla\varphi^{T,g}_s(x)+(1-\theta)\nabla\varphi^{T,g'}_s(x):\theta\in[0,1]\}\ees
such that
\bes
\begin{split}
&H(x,\nabla\varphi^{T,g'}_s(x))-H(x,\nabla\varphi^{T,g}_s(x))-D_pH(x,\nabla\varphi^{T,g}_s(x))\cdot\nabla\psi_s(x)\\
&= (\nabla\psi_s(x))^{\top}\cdot \frac{1}{2}D_{pp}H(x,v_s(x))\cdot\nabla\psi_s(x)\\
&:=\gamma_s(x) \cdot\nabla \psi_s(x).
\end{split}
\ees
Moreover, combining \eqref{eq:SP_grad_est_3} with \eqref{eq:Hess_Ham_bound} we obtain
\be\label{eq:SP_stab_fin_cond_3}
\sup_{\substack{x\in\bbRD \\ t\leq s\leq T}}|\gamma_s(x)|\leq \frac{M_u^2}{2 \omega_{ M^{\varphi,g}_x\vee M^{\varphi,g'}_x}} \,(M^{\varphi,g}_x+M^{\varphi,g'}_x).
\ee
As a consequence of Proposition \ref{item_1:SP_opt_cond} $\psi$ is a classical solution of
\bes
    \begin{cases}
    \partial_s \psi_s(x)+(-\gamma_s(x)+\beta^{T,g}_s(x))\cdot\nabla\psi_s(x) +\frac{\sigma^2}{2}\Delta\psi_s(x)=0, \,(s,x) \in[t,T)\times \bbRD,\\
    \psi_T(x)=(g'-g)(x), \quad x\in\bbRD.
    \end{cases}
\ees
 Interpreting the above as a Kolmogorov equation $(\partial_s +\cL_s)\psi=0$ for the Markov generator
 \bes
 \cL_s f(x)  = \frac{\sigma^2}{2}\Delta f(x) +\Big(\beta^{T,g}_s-\gamma_s\Big)(x)\cdot\nabla f(x)
 \ees
 we obtain the following probabilistic representation for $\psi_t(x)$:
 \bes
 \psi_t(x)=\bbE[(g'-g)(Y^{t,x}_T)], 
 \ees
where
 \bes
 \begin{cases} 
 \De Y^{t,x}_s= \tilde\beta_s( Y^{t,x}_s)\De s + \sigma\De B_s, \quad s\in[t,T],\\
 Y^{t,x}_t=x.
 \end{cases}
 \ees
and
\bes
\tilde\beta_s(x)=\beta^{T,g}_s(x)-\gamma_s(x) \quad \forall x\in\bbRD, t\leq s \leq T.
\ees
Fix now $r>0,t\leq s \leq T$. We have 
 \be\label{eq:SP_stab_fin_cond_5}
 \begin{split}
\kappa_{\tilde\beta_s}(r) &\stackrel{\eqref{eq:SP_stab_fin_cond_3}}{\geq} \kappa_{\beta^{T,g}_s}(r)-\frac{ M_u^2}{ \sigma^2\omega_{ M^{\varphi,g}_x\vee M^{\varphi,g'}_x} \, r}   \\
&\stackrel{\eqref{eq:contraction_SP_3}}{\geq} \bar\kappa_b(r)-\frac{M_u^2}{\sigma^2 r}\Big(\frac{4(1+M_x^{\varphi,g})}{\omega_{M_x^{\varphi,g}}}+\frac{M^{\varphi,g}_x+M^{\varphi,g'}_x}{ \omega_{M^{\varphi,g}_x \vee M^{\varphi,g'}_x}} \,\Big)=\kappa(r).
\end{split}
 \ee
 We can now conclude thanks to Proposition \ref{item_1:contraction_coup_by_ref}. Indeed, for all $x,x'$ we have
  \bes
  \begin{split}
  \psi_t(x)-\psi_t(x')&=\bbE\big[(g'-g)(Y^{t,x}_T)\big]-\bbE\big[(g'-g)(Y^{t,x'}_T)\big]\\ &\leq \|g'-g\|_{f} W_f(\mathrm{Law}(Y^{t,x}_T),\mathrm{Law}(Y^{t,x'}_T))\\ &\leq \|g'-g\|_{f} |x-x'| e^{-\lambda(T-t)}
  \end{split}
  \ees
  with $\lambda,f$ as in \eqref{eq:SP_stab_fin_cond_1}. The bound \eqref{eq:SP_stab_fin_cond_4} is now proven and \eqref{eq:SP_stab_fin_cond_2} follows immediately from \eqref{eq:Lyapunov_funct_coup_by_ref_2}. We now proceed to remove the assumption that $g,g'\in C^3(\bbRD)$. To this aim, recall that for any $g,g'\in C^1_{\lip}(\bbRD)$ (see e.g. \cite[Prop. A item (b)]{czarnecki2006approximation}) there exist sequences $(g^M)_{M\in\N},(g'^M)_{M\in\N}\subseteq C^3_{\lip}(\bbRD)$ with the property that 
  \be\label{eq:SP_stab_fin_cond_9}
    \lim_{M\rightarrow+\infty} \sup_{x\in\bbRD}| g^M-g|(x)+ |g'^M-g'|(x)+ | \nabla g^M-\nabla g|(x)+ |\nabla g'^M-\nabla g'|(x)=0.
  \ee
  Fix now $\varepsilon>0$. Note that
  \be\label{eq:SP_stab_fin_cond_7} (1-\varepsilon)\|g^M \|_{\mathrm{Lip}}\leq \|g\|_{\mathrm{Lip}},   \quad (1-\varepsilon)\|g'^M \|_{\mathrm{Lip}}\leq \|g'\|_{\mathrm{Lip}} 
  \ee
  for all $M$ large enough.
Define now $\kappa^{M,\varepsilon}$ by replacing $g,g'$ in \eqref{eq:SP_stab_fin_cond_1} with $(1-\varepsilon)g^M$ and $(1-\varepsilon)g'^M$ respectively. Because of \eqref{eq:SP_stab_fin_cond_7} we have that $\kappa^{M,\varepsilon}\geq \kappa$ for all $M,\varepsilon$, where $\kappa$ is as in \eqref{eq:SP_stab_fin_cond_1}. But then, leveraging  the monotonicity properties of Proposition \ref{item_3:Lyapunov_funct_coup_by_ref} we find that for all $M,\varepsilon$
\begin{equation}\label{eq:SP_stab_fin_cond_8}
     \|\varphi^{T,(1-\varepsilon)g^M}_t-\varphi^{T,(1-\varepsilon)g'^M}_t\|_{f} \leq (1-\varepsilon)\|g^M-g'^M\|_{f}  \exp(-\lambda (T-t)).
\end{equation}
where $f=\bmf(\kappa),\lambda=\bm\lambda(\kappa)$. Letting $M\rightarrow+\infty$ in the above and using \eqref{eq:SP_stab_fin_cond_9} which in particular implies that $\|g^M-g\|_{f}+\|g'^M-g'\|_{f}$ converges to $0$ as well as the pointwise convergence  of $\varphi^{T,(1-\varepsilon)g^M}_t$ (resp. $\varphi^{T,(1-\varepsilon)g'^M}_t$) to $\varphi^{T,(1-\varepsilon)g}_t$ (resp.$\varphi^{T,(1-\varepsilon)g'}_t$), we find 
 \bes
 \|\varphi^{T,(1-\varepsilon)g}_t-\varphi^{T,(1-\varepsilon)g'}_t\|_{f} \leq (1-\varepsilon)\|g-g'\|_{f}  \exp(-\lambda (T-t)).
\ees
If we now let $\varepsilon\rightarrow0$ and use the fact that $\varphi^{T,(1-\varepsilon)g}_t$ and $\varphi^{T,(1-\varepsilon)g'}_t$ converge pointwise to $\varphi^{T,g}_t$ and $\varphi^{T,g'}_t$, the desired conclusion \eqref{eq:SP_stab_fin_cond_4} follows.

\end{proof}

\begin{lemma}\label{lem:sticky_coup}
Let Assumption \ref{ass:SP_wellposed} -\ref{ass:SP} hold. Moreover, assume that $g,g'\in C^1_{\lip}(\bbRD)$. Then the estimate
\be\label{eq:sticky_coup_1}
W_1(\mathrm{Law}(\bmX^{0,x,T,g}_s),\mathrm{Law}(\bmX^{0,x,T,g'}_s))  \leq \overleftarrow{C}\|g-g'\|_{\mathrm{Lip}} \exp(-\overleftarrow{\lambda}(T-s)).
\ee
holds uniformly on $x\in\bbRD$ and $0\leq s\leq T$ where
\be\label{eq:sticky_coup_4}
\overleftarrow{\lambda}=\bm\lambda(\kappa), \quad \overleftarrow{C}= \frac{M^2_{u}}{\lambda\,\bmC(\kappa)\,\omega_{\,M_x^{\varphi,g}\vee M^{\varphi,g'}_x}},
\ee
and $\kappa$ is given by \eqref{eq:SP_stab_fin_cond_1}.
\end{lemma}

\begin{proof}
From Proposition \ref{prop:SP_opt_cond} we know that $(\bmX^{0,x,T,g}_s)_{s\in[0,T]}$ is a solution of
\eqref{eq:SDE_1} for the drift field $\beta^{T,g}_{\cdot}(\cdot)$ Likewise, $(\bmX^{0,x,T,g'}_s)_{s\in[0,T]}$ is a solution of \eqref{eq:SDE_1} for the drift field $\beta^{T,g'}_{\cdot}(\cdot)$. We have for any $0\leq s \leq T,x\in\bbRD$
\be\label{eq:sticky_coup_3}
    \begin{split}
         | \beta^{T,g}_s(x)-\beta^{T,g'}_s(x)|&\stackrel{\eqref{eq:Hess_Ham_bound}}{\leq} \frac{M^2_{u}}{\omega_{M_x^{\varphi,g}\vee M^{\varphi,g'}_x}}  |\nabla\varphi^{T,g}_{s}(x)-\nabla\varphi^{T,g'}_{s}(x)|\\
         &\stackrel{\eqref{eq:SP_stab_fin_cond_2}}{\leq} \,\frac{M^2_{u}}{\bmC(\kappa)\,\omega_{M_x^{\varphi,g}\vee M^{\varphi,g'}_x}} \, \|g-g'\|_{\mathrm{Lip}}\exp(-\bm\lambda(\kappa)(T-s)):=M_s,
    \end{split}
\ee

where $\kappa$ is defined by \eqref{eq:SP_stab_fin_cond_1}. Consider $\kappa_0$ as given by \eqref{eq:contraction_SP_3} and define for any $\varepsilon,L>0$ the functions
\bes
\begin{split}
\kappa_{0,L}(r)=\min\{\kappa_0(r), \inf_{[L,+\infty]}\kappa_0\},\\
\kappa_{0,\varepsilon,L}:=\min\{\max\{ \kappa_0(r),-2\sigma^{-2}(M_{xp}^{H,g}-M^{\varphi,g}_{xx,\varepsilon}\,M_u^2\,\omega^{-1}_{M_x^{\varphi,g}})\}, \inf_{[L,+\infty]}\kappa_0 \}.
\end{split}
\ees
From Corollary \ref{cor:relax_boundary} and the bounds \eqref{eq:coeff_bound_2}-\eqref{eq:coeff_bound_3} and \eqref{eq:SP_Hess_est_2} we deduce that
\bes
 \inf_{s\in[0,T-\varepsilon]}\kappa_{\beta^{T,g}_s} \geq \kappa_{0,\varepsilon,L} \quad\forall \varepsilon,L>0.
\ees
Moreover, $\kappa_{0,\varepsilon,L}$ is globally Lipschitz and constantly equal to a positive constant outside an interval if $L$ is large enough.  We can now invoke  Theorem \ref{thm:sticky_coupl} to obtain existence of a coupling  $(\bar\bmX^{g}_s,\bar{\bmX}^{g'}_s)_{s\in[0,T]}$ of $\bmX^{0,x,T,g}_{\cdot}$ and $\tilde{\bmX}^{0,x,T,g'}_{\cdot}$ and a process $(r_s)_{s\in[0,T-\varepsilon]}$ such that
\be\label{eq:sticky_coup_2}
\text{almost surely} \quad |\bar\bmX^{g}_s-\bar{\bmX}^{g'}_s|\leq r_s \quad \forall s\in[0,T-\varepsilon]
\ee
and $(r_s)_{s\in[0,T-\varepsilon]}$ is a solution to the SDE
\bes
\De r_s= \big(M_{s} -\kappa_{0,\varepsilon,L}(r_s)r_s)\De s + 2\sigma\mathbf{1}_{\{r_s>0\}}\De W_s, \quad r_0=0,
\ees
where $W_s$ is a one-dimensional Brownian motion. Using Itô-Tanaka formula as in \cite[Lemma 25]{eberle2016reflection} we find that, setting $ f_{0,\varepsilon,L}=\bmf(\kappa_{0,\varepsilon,L})$,
\bes
\begin{split}
\De f_{0,\varepsilon,L}(r_s)&=  f'_{0,\varepsilon,L}(r_s)\big(M_s-\kappa_{0,\varepsilon,L}(r_s) r_s )\De s + 2\sigma^2 f''_{0,\varepsilon,L}(r_s)\mathbf{1}_{\{r_s>0\}} + 2\sigma  f'_{0,\varepsilon,L}(r_s)\mathbf{1}_{\{r_s>0\}}\De W_s\\
&\stackrel{\text{Prop.}\ref{item_3:Lyapunov_funct_coup_by_ref}}{\leq}  f'_{0,\varepsilon,L}(r_s)M_s\De s + 2\sigma  f'_{0,\varepsilon,L}(r_s)\mathbf{1}_{\{r_s>0\}}\De W_s
\end{split}
\ees
But then, taking expectations we obtain that for all $s\leq T-\varepsilon$:
\bes
\begin{split}
\bbE[f_{0,\varepsilon,L}(r_s)] &\leq \int_0^s M_{\theta}\bbE[ f'_{0,\varepsilon,L}(r_\theta)]\De \theta\\ &\stackrel{\eqref{eq:Lyapunov_funct_coup_by_ref_2}}{\leq} \int_0^s M_\theta \De \theta \\
&\stackrel{\eqref{eq:sticky_coup_3}}{\leq} \frac{M^2_{u}}{\bmC(\kappa)\,\omega_{M_x^{\varphi,g}\vee M^{\varphi,g'}_x}}\|g-g'\|_{\mathrm{Lip}}\bm\lambda^{-1}(\kappa)\exp(-\bm\lambda(\kappa)(T-s))
\end{split}
\ees
At this point, we can use \eqref{eq:sticky_coup_2} and \eqref{eq:Lyapunov_funct_coup_by_ref_2} to conclude that for any $\varepsilon,L$ and $s\leq T-\varepsilon$:
\bes
W_1(\mathrm{Law}(\bmX^{0,x,T,g}_s),\mathrm{Law}(\bmX^{0,x,T,g'}_s))\leq  \frac{M^2_{u}\|g-g'\|_{\mathrm{Lip}}\exp(-\bm\lambda(\kappa)(T-s))}{\bm\lambda(\kappa)\bmC(\kappa)\bmC(\kappa_{0,\varepsilon,L})\,\omega_{M_x^{\varphi,g}\vee M^{\varphi,g'}_x}}.
\ees
In order to reach the conclusion, it remains to show that $\sup_{L,\varepsilon>0}\bmC(\kappa_{0,\varepsilon,L})\geq \bmC(\kappa)$. To do so, we begin observing that the the monotonicity of $\bmC$ implies $\bmC(\kappa_{0,\varepsilon,L})\geq \bmC(\kappa_{0,L})$ for any $L,\varepsilon>0$. Finally, by construction $\lim_{L\rightarrow+\infty}\bmC(\kappa_{0,L})=\bmC(\kappa_0)\geq \bmC(\kappa)$, where to obtain the last inequality we used again the monotonicity of the map $\bmC$. 
\end{proof}
We now are in position to extend the scope of Proposition \ref{prop:SP_opt_cond}.
\begin{prop}\label{prop:policy_opt}
Let Assumption \ref{ass:SP_wellposed}-\ref{ass:SP} hold and $g\in C^1_{\lip}(\bbRD)$. Then the map 
\bes
(s,x)\mapsto w(x,\nabla\varphi^{T,g}_s(x))
\ees
is an optimal Markov policy in the sense of Proposition \ref{item_4:SP_opt_cond}.
\end{prop}
The proof of this proposition cannot be done through the classical verification argument since we do not know that $\varphi^{T,g}$ is twice continuously differentiable. However, we know that its gradient is Lipschitz and this will be enough to conclude. We defer the proof to the appendix section.

\begin{cor}\label{cor:relax_boundary}
Let Assumption \ref{ass:SP_wellposed}-\ref{ass:SP} hold and $g\in C^1_{\lip}(\bbRD)$. Then, the conclusion \eqref{eq:contraction_SP_3}-\eqref{eq:contraction_SP_1_b} of Theorem \ref{thm:contraction_SP} hold.
\end{cor}
\begin{proof}
Consider an approximating $g^M\subseteq C^3_{\lip}(\bbRD)$ such that 
\bes 
\lim_{M\rightarrow +\infty} \sup_{x\in\bbRD}|g^M-g|(x)=0,\quad  \|g^M \|_{\mathrm{Lip}}\leq \|g\|_{\mathrm{Lip}} \quad \forall M.
\ees
Since $\|g^M \|_{\mathrm{Lip}}\leq \|g\|_{\mathrm{Lip}}$, then \eqref{eq:contraction_SP_1_b} holds for any $M$  with constants $\overrightarrow{C}_M,\overrightarrow{\lambda}_M$ that are better than those needed for the conclusion by monotonicity of $\bmC(\cdot)$ and $\bm\lambda(\cdot)$. Using this observation, letting $M\rightarrow+\infty$, and leveraging the stability estimate of Lemma \ref{lem:sticky_coup} to obtain that for any $0\leq s \leq T$ and $\xi$ with $\mathrm{Law}(\xi)\in\cP_1(\bbRD)$
\bes
\begin{split}
\lim_{M\rightarrow+\infty} W_1(\mathrm{Law}(\bmX^{0,x,T,g}_s),\mathrm{Law}(\bmX^{0,x,T,g^M}_s))= 0,
\end{split}
\ees
we obtain the desired conclusion.
\end{proof}

\subsection{The ergodic problem}

In the following Lemma we appeal at the notion of stationary solution to equation \eqref{eq:SP_HJB}. By this, we mean a pair $(\alpha,\varphi)$ such that $\varphi(0)=0$ and $(t,x)\mapsto\varphi(x)+\alpha(T-t)$ is a viscosity solution of \eqref{eq:SP_HJB} for the boundary condition $g=\varphi$. For a definition of viscosity solutions, we refer to \cite[Def 5.1]{yong1999stochastic}.

\begin{lemma}\label{lem:erg_HJB}
Assume that Assumption \ref{ass:SP_wellposed}-\ref{ass:SP} hold. Then there exist a unique pair $(\alpha^{\infty},\varphi^{\infty})$ in $\bbR \times C^1_{\lip}(\bbRD)$ such that $(\alpha^{\infty},\varphi^{\infty})$ is a stationary viscosity solution to \eqref{eq:SP_HJB} for all $T>0$. Moreover, $\|\nabla\varphi^{\infty}\|_{\mathrm{Lip}}<+\infty$ and $\|\varphi^{\infty} \|_{\mathrm{Lip}}\leq M^{\varphi,0}_x$.
\end{lemma}
\begin{proof}
We recall that $C(\bbRD)$ is the set of continuous functions and that $C_0(\bbRD)=\{ g\in C(\bbRD):g(0)=0\}$. We equip $C_0(\bbRD)$ with the topology of uniform convergence on compact sets and for any $M$ we define
\bes 
B_M=\{g\in C_0(\bbRD) : \|g\|_f\leq M\}, \quad f=\bmf(\bar\kappa_b).
\ees 
Since $f$ is equivalent to the identity function, Arzéla-Ascoli Theorem implies that $B_M$ is a convex compact subset of $C_0(\bbRD)$. Next, for any $T,M>0$ we consider the mapping 
\bes 
\Phi_T:B_M\longrightarrow C(\bbRD), \quad g \mapsto \varphi^{T,g}_0
\ees
and the mapping 
\be\label{eq:erg_HJB_4}
\bar\Phi_T:B_M\longrightarrow C_0(\bbRD), \quad g \mapsto \varphi^{T,g}_0(\cdot)-\varphi^{T,g}_0(0).
\ee
Thanks to \eqref{eq:SP_grad_est_6} we know that
\bes
\bar\Phi_T(B_M) \subseteq B_M \quad \forall \, M \geq \frac{M^F_x}{C\lambda}=M^{\varphi,0}_{x},
\ees
where $C=\bmC(\bar\kappa_b),\lambda=\bm\lambda(\bar\kappa_b)$. We shall prove at Lemma \ref{lem:rest_cont} that the restriction of $\bar\Phi_T$ to $B_M$ is continuous. We have thus verified the hypothesis of Schauder's fixed point theorem whose application yields the existence of a fixed point $\varphi^{\infty,T}$ for $\bar\Phi_T$ such that $\|\varphi^{\infty,T}\|_{\lip}\leq \|\varphi^{\infty,T}\|_{f}\leq M^{\varphi,0}_{x}$. But then, setting $\alpha^{\infty,T}:=\Phi_T(\varphi^{\infty,T})(0)$, we deduce that  \bes
\Phi_T(\varphi^{\infty,T})=\varphi^{\infty,T}+\alpha^{\infty,T}.
\ees 
Thanks to Proposition \ref{cor:reg_eff} we also know that $\varphi^{\infty,T}\in C^1_{\mathrm{Lip}}(\bbRD)$ and that $\nabla\varphi^{\infty,T}\in\mathrm{Lip}(\bbRD;\bbRD).$ Moreover, because of the bound \eqref{eq:SP_stab_fin_cond_4}, for any $T>0$ the fixed point is unique in the set $C_0(\bbRD)\cap C^1_{\lip}(\bbRD)$ \footnote{Note that the function $f$ used here does not coincide with the function $f$ appearing in \eqref{eq:SP_stab_fin_cond_4}. However, both functions are equivalent to the identity function and therefore the corresponding norms are all equivalent to $\|\cdot\|_{\mathrm{Lip}}$.} Next, we observe that, because of the dynamic programming principle we have that, for any fixed $T$ 
\bes
\Phi_{T}(g) = \Phi_{T/2}(\Phi_{T/2}(g)), \quad \forall g\,\text{s.t.}\,\, \| g\|_f <+\infty.
\ees
This fact, together with the uniqueness of the fixed point imply that 
\bes
\alpha^{\infty,T} = 2\alpha^{\infty,T/2}, \quad \varphi^{\infty,T}=\varphi^{\infty,T/2}, \quad \forall T>0
\ees
and therefore that
\bes
\Phi^{T/2}(\varphi^{\infty,T})=\varphi^{\infty,T}+\frac{\alpha^{\infty,T}}{2}\quad \forall T>0.
\ees
Iterating this argument, we find that
\be\label{eq:erg_HJB_1}
\Phi_T(\varphi^{\infty,1}) = \varphi^{\infty,1} - \alpha^{\infty,1} T:= \varphi^{\infty} - \alpha^{\infty} T\quad \forall T\in D,
\ee
where $D$ is the set of dyadic numbers in $[0,1]$. Given the continuity of the map $T\mapsto \Phi_T(\varphi^{\infty})(x)$ for all $x\in\bbRD$, that is also a consequence of the dynamic programming principle, we can extend \eqref{eq:erg_HJB_1} to all $T\in[0,1]$. Since $\Phi_{1-t}(\varphi^{\infty})=(\varphi^{1,\varphi_{\infty}}_{t})_{0\leq t\leq 1}$ and the latter is a viscosity solution of \eqref{eq:SP_HJB}, see Lemma \ref{cor:reg_eff}, we have shown the existence of a stationary solution for $T=1$, and the same proof can be used to argue for a general $T$. We now proceed to prove uniqueness. To this aim, assume that  $(\bar\alpha^{\infty},\bar\varphi^{\infty})\in \bbR\times C^1_{\lip}(\bbRD)$ is another stationary viscosity solution with $\bar\varphi^{\infty}(0)=0$ and set $M=\|\bar\varphi^{\infty}\|_{\lip}$. If we now set 
\bes
R_0=\frac{M_u(1+M)}{\omega_M},\quad H^{R}(x,p) = \sup_{|u|\leq R}b(x,u)\cdot p+F(x,u)
\ees
we find that, arguing exactly as in the proof of \eqref{eq:control_bound_1} 
\bes
   H(x,p)=H^{R}(x,p) \quad \forall x\in\bbRD, |p|\leq M, R\geq R_0.
\ees
But then, since $\bar\varphi^{\infty}$ is of class $C^1(\bbRD)$, it follows from the definition of viscosity solution that $(\bar\alpha^{\infty},\bar\varphi^{\infty})$ is also a stationary viscosity solution for the equation 
\begin{equation}\label{eq:erg_HJB_2}
    \begin{cases}
    \partial_t \varphi_t(x) -H^{R}(x,\nabla \varphi_t(x))+\frac{\sigma^2}{2}\Delta\varphi_t(x)=0 \quad (t,x) \in(0,1)\times \bbRD,\\
    \varphi_1(x)=\bar\varphi^{\infty}(x) \quad x\in\bbRD
    \end{cases}
\end{equation}
for any $R\geq R_0$. For any $R>0$ uniqueness of viscosity solution for the above equation is well known under the current assumptions on $b(\cdot,\cdot)$ and $F(\cdot,\cdot)$, see \cite[Thm 6.1]{yong1999stochastic} for example. Moreover,
\be\label{eq:erg_HJB_3}
\forall R>R_0, \quad \bar\varphi^{\infty}+\bar\alpha^{\infty}(1-t) = \inf_{u\in\cU^R_{[t,1]}} J_{t,x}^{T,\bar\varphi^{\infty}}(u)
\ee
where 
\bes
\cU^R_{[t,1]} = \left\{(u_s)_{s\in[t,1]} \in \cU_{[t,1]}\,\, \text{s.t.} \,\,\,|u|\leq R\,\, \bbP-\text{a.s.}\right\}
\ees
By eventually letting $R\rightarrow+\infty$ in \eqref{eq:erg_HJB_3} and using Assumption \ref{ass:SP_wellposed} and \ref{ass:SP} we find that 
\bes
\bar\varphi^{\infty}(x)+\alpha^{\infty}(1-t) = \varphi^{1,\bar\varphi^{\infty}}_t(x), \quad (t,x)\in[0,T]\times\bbRD.
\ees
But then, we find that $\bar\varphi^{\infty}$ is a fixed point in $C^1_{\lip}(\bbRD)$ of the operator $\bar\Phi_{1}$. The contraction estimate \eqref{eq:SP_stab_fin_cond_4} guarantees uniqueness of such a fixed point. Therefore $\bar\varphi^{\infty}=\varphi^{\infty}$ and $\bar\alpha^{\infty}=\alpha^{\infty}$. It only remains to show that $\|\nabla\varphi^{\infty}\|_{\mathrm{Lip}}<+\infty$, that is a direct consequence of Corollary \ref{cor:reg_eff}.
\end{proof}

\subsection{Proof of Theorem \ref{thm:SP_val_fun} and \ref{thm:SP_turnpike_1}}

\begin{proof}[Proof of Theorem \ref{thm:SP_val_fun}]
The proof of item $(i)$ is given at Lemma \ref{lemma:SP_grad_est}, the proof of item $(ii)$ is obtained from Proposition \ref{cor:reg_eff}. $(iii)$ and $(iv)$ are proven at Lemma \ref{lem:erg_HJB} and \ref{lemma:SP_stab_fin_cond} respectively.
\end{proof}

\begin{proof}[Proof of Theorem \ref{thm:SP_turnpike_1}]
Combining through a triangular inequality the contraction estimate \eqref{eq:contraction_SP_1_b} from Theorem \ref{thm:contraction_SP} with the sensitivity bound \eqref{eq:sticky_coup_1} from Lemma \ref{lem:sticky_coup} we obtain the proof of \ref{item_1:SP_turnpike_1} for $\mathrm{Law}(\xi)=\delta_x,\mathrm{Law}(\xi')=\delta_{x'}$, $x,x'\in\bbRD$. The extension to arbitrary random variables $\xi,\xi'$ is obtained in a classical way averaging over the initial condition, see e.g. the proof of Corollary 2 in \cite{eberle2016reflection}.  We now turn to the proof of $(ii)$. From Lemma \ref{lem:erg_HJB} we know that $\varphi^{\infty}$ is of class $C^1(\bbRD)$and $\nabla\varphi^{\infty}$ is a bounded Lipschitz vector field. But then, it follows from Lemma \ref{lem:coeff_bound} and in particular from \eqref{eq:Hess_Ham_bound} and \eqref{eq:coeff_bound_2} that
\bes
\bbRD\ni x \mapsto -D_pH(x,\nabla\varphi^{\infty}(x))=\beta^{\infty}(x)
\ees
is a Lipschitz-continuous vector field, from which existence of strong solutions for \eqref{eq:SP_turnpike_2} follows. Moreover, we observe that the fixed point property of $\varphi^{\infty}$ implies that for all $T>0$
\be\label{eq:turnpike_9}
\beta^{\infty}(x)= \beta^{T,\varphi^{\infty}}_s(x),\quad \forall (s,x)\in[0,T]\times\bbRD
\ee
But then, thanks to Corollary \ref{cor:relax_boundary} we have that $\beta^{\infty}\in K$. At this point, the existence and uniqueness of an invariant measure $\mu^{\infty}$ in $\cP_1(\bbRD)$ is obtained in a straightforward way from the contraction estimate at \cite[Thm.\,1]{eberle2016reflection}. Let us now prove item $(iii)$. To this aim, observe that \eqref{eq:turnpike_9} implies that the unique strong solution of \eqref{eq:SP_turnpike_2} is precisely $(\bmX^{0,x,T,\varphi^{\infty}}_s)_{s\in[0,T]}$. We define 
\be\label{eq:SP_turnpike_13}
\tau= \frac{\log\big( M_x^g/M^{\varphi,0}_x\big)}{\bm\lambda(\bar\kappa_b)}, \quad g'_{\tau}=\varphi^{T,g'}_{T-\tau}
\ee
The gradient estimate \eqref{eq:SP_grad_est_7} implies that 
\be\label{eq:SP_turnpike_10}
\|g'_{\tau}\|_{\mathrm{Lip}}\leq 2 M_{x}^{\varphi,0}.
\ee
We have, using the dynamic programming principle and the properties of stationary solutions:
\bes
W_1(\mathrm{Law}(\bmX^{0,x,T,\varphi^{\infty}}_s),\mathrm{Law}(\bmX^{0,x',T,g'}_s))=W_1(\mathrm{Law}(\bmX^{0,x,T-\tau,\varphi^{\infty}}_s),\mathrm{Law}(\bmX^{0,x',T-\tau,g'_{\tau}}_s))
\ees
At this point, we can use the bound \eqref{eq:SP_turnpike_8} choosing $g=\varphi^{\infty}$, $g'=g'_{\tau}$ and shortening the time horizon, i.e. replacing $T$ with $T-\tau$. We obtain, with the help of \eqref{eq:SP_turnpike_10} and $\|\varphi^{\infty}\|_{\lip}\leq M^{\varphi,0}_x $ that
\be\label{eq:SP_turnpike_11}
\begin{split}
W_1(\mathrm{Law}(\bmX^{0,\xi,T-\tau,\varphi^{\infty}}_s),\mathrm{Law}(\bmX^{0,\xi',T-\tau,g'_{\tau}}_s)) \leq C\, W_1(\mathrm{Law}(\xi),\mathrm{Law}(\xi'))\exp(-\lambda^{\infty} s)\\
+ C\,\|\varphi^{\infty}-g'_{\tau}\|_{\mathrm{Lip}} \exp(-\lambda^{\infty} (T-\tau-s)\big)
\end{split}
\ee
holds uniformly on $0\leq s\leq T-\tau$ and $\mathrm{Law}(\xi),\mathrm{Law}(\xi')\in\cP_1(\bbRD)$ with 
\be\label{eq:SP_turnpike_15}
\lambda^{\infty}=\bm\lambda(\kappa), \quad \kappa(r)=\bar\kappa_b(r)-\frac{M_u^2}{\sigma^2 r}\Big(\frac{4(1+M_x^{\varphi,0})}{\omega_{M_x^{\varphi,0}}}+\frac{3M^{\varphi,0}_x}{ \omega_{2M^{\varphi,0}_x}} \,\Big)\quad \forall r>0
\ee
 and 
 \bes
C=\max\left\{\frac{M^2_{u}}{\lambda\,\bmC(\kappa)\,\omega_{2 M^{\varphi,0}_x}},\frac{1}{\bmC(\kappa)} \right\},
 \ees
 where we used the monotonicity of the function $\bmC(\cdot)$ to obtain the last expression.
Note in particular that $\lambda^{\infty}$ does not depend on $g$. To conclude, we first invoke \eqref{eq:SP_stab_fin_cond_2} to obtain  
\be\label{eq:SP_turnpike_12}
\|\varphi^{\infty}-g'_{ \tau}\|_{\mathrm{Lip}}\leq \tilde{C}\exp(-\tilde{\lambda} \tau)\|\varphi^{\infty}-g'\|_{\lip}
\ee
with 
\be\label{eq:SP_turnpike_18}
(\tilde{C},\tilde{\lambda})=(\bmC^{-1}(\tilde\kappa),\bm\lambda(\tilde{\kappa})),\quad  \tilde{\kappa}(r)=\bar\kappa_b(r)-\frac{M_u^2}{\sigma^2 r}\Big(\frac{4(1+M_x^{\varphi,0})}{\omega_{M_x^{\varphi,0}}}+\frac{M^{\varphi,0}_x+M^{\varphi,g'}_x}{ \omega_{M^{\varphi,g'}_x}} \,\Big)\quad \forall r>0.
\ee
The desired result now follows plugging \eqref{eq:SP_turnpike_12} into \eqref{eq:SP_turnpike_11} setting
\be\label{eq:SP_turnpike_16}
A= \max\{C,C\,\tilde{C}\exp((\lambda^{\infty}-\tilde\lambda) \tau) \}
\ee
and choosing $\mathrm{Law}(\xi)=\mu^{\infty}$. All what is left to do is to prove item $(iv)$. To this aim, observe that, using Lemma \ref{lemma:SP_stab_fin_cond} we have that for all $s\leq T-\tau$
\bes
\begin{split}
     \|\varphi^{\infty}-\varphi^{T-\tau,g'_{\tau}}_s \|_{\lip}
    &\leq \bmC(\kappa)^{-1} \exp(-\lambda^{\infty}(T-\tau-s))\| \varphi^{\infty}-g'_{\tau}\|_{\lip}\\
    &\stackrel{\eqref{eq:SP_turnpike_12}}{\leq}\tilde{C}\exp((\lambda^{\infty}-\tilde{\lambda}) \tau)\bmC(\kappa)^{-1}\|\varphi^{\infty}-g'\|_{\lip} \exp(-\lambda^{\infty}(T-s))
\end{split}
\ees
where $\lambda^{\infty},\kappa$ have been defined at \eqref{eq:SP_turnpike_15} and $(\tilde{\lambda},\tilde{C})$ at \eqref{eq:SP_turnpike_18}. In the next lines, we shall use the notation $\lesssim$ to denote inequality up to a positive multiplicative constant depending only on the constants $M_x,M_u,M_{xx},M_{xu},\|g'\|_{\lip}$ and the functions $\omega,\bar\kappa_b$. In particular, we rewrite the above as 
\be\label{eq:SP_turnpike_19}
\|\varphi^{\infty}-\varphi^{T,g'}_s \|_{\lip}\lesssim  \|\varphi^{\infty}-g'\|_{\lip} \exp(-\lambda^{\infty}(T-s))
\ee
Consider now $\xi,\xi'$ with $\rm{Law}(\xi)=\mu^{\infty},\mathrm{Law}(\xi')\in\cP_1(\bbRD)$ and consider the corresponding optimal controls
\bes
\begin{split}
\bmu^{0,x,T,g}_s=w(\bmX^{0,\xi',T,g'},\nabla\varphi^{T,g'}_s(\bmX^{0,\xi',T,g'}))\\
\bmu^{0,x,T,\varphi^{\infty}}_s=w(\bmX^{0,\xi,T,\varphi^{\infty}},\nabla\varphi^{\infty}(\bmX^{0,\xi,T,\varphi^{\infty}}_s)))
\end{split}
\ees
In what follows, since there is no ambiguity, we shall denote$\bmX^{0,\xi’,T,g’}_s$ by $\bmX'_s$ and the process $\bmX^{0,\xi,T,\varphi^{\infty}}_s$  by $\bmX_s$.
\be\label{eq:SP_turnpike_20}
\begin{split}
&W_1(\mathrm{Law}(w(\bmX'_s,\nabla\varphi^{T,g'}_s(\bmX'_s))),\mathrm{Law}(w(\bmX_s,\nabla\varphi^{\infty}_s(\bmX_s))))\\
&\stackrel{\eqref{eq:SP_turnpike_19},\eqref{eq:coeff_bound_13}}\lesssim  \|\varphi^{\infty}-g'\|_{\lip}\exp(-\lambda^{\infty}(T-s))\\
&+W_1\Big(\mathrm{Law}(w(\bmX'_s,\nabla\varphi^{\infty}(\bmX'_s))),\mathrm{Law}(w(\bmX_s,\nabla\varphi^{\infty}(\bmX_s)))\Big)
\end{split}
\ee
Next, we observe that 
\bes
\begin{split}&W_1\Big(\mathrm{Law}(w(\bmX'_s,\nabla\varphi^{\infty}(\bmX'_s))),\mathrm{Law}(w(\bmX_s,\nabla\varphi^{\infty}(\bmX_s)))\Big)\\
&\stackrel{\text{Lemma}\, \ref{lemma:SP_grad_est},\eqref{eq:coeff_bound_11},\eqref{eq:coeff_bound_13}}{\lesssim}   W_1(\mathrm{Law}(\bmX'_s),\mathrm{Law}(\bmX_s))+ W_1(\mathrm{Law}(\nabla\varphi^{\infty}(\bmX'_s)),\mathrm{Law}(\nabla\varphi^{\infty}(\bmX_s)))\\
& \stackrel{\nabla\varphi^{\infty}\in\lip(\bbRD;\bbRD)}{\lesssim}  W_1(\mathrm{Law}(\bmX'_s),\mathrm{Law}(\bmX_s)))\\
&\stackrel{\text{Theorem} \ref{item_3:SP_turnpike}}{\lesssim}  \|\varphi^{\infty}-g'\|_{\lip}\exp(-\lambda^{\infty}(T-s)) + W_1(\mu^{\infty},\mathrm{Law}(\xi'))\exp(-\lambda^{\infty} s)
\end{split}
\ees
Plugging this bound into \eqref{eq:SP_turnpike_20} gives the desired result.
\end{proof}

\section{Appendix}


\subsection{Optimality conditions}

\begin{proof}[Proof of Proposition \ref{prop:SP_opt_cond}]

We first show the existence statement of Proposition \ref{item_1:SP_opt_cond} that is a consequence of known results and an approximation procedure. Fix $M\in\N$ and consider the drift field $b^M:\bbR^{d+p}\longrightarrow\bbRD$ defined as follows:
\bes
b^M(x,u) = b(x,0)+ (b(x,u)-b(x,0)) \chi^M(|u|),
\ees
where $\chi^M$ is a smooth decreasing function such that  
\bes
\chi^M(r) = \begin{cases}
	1, & \text{if $r\leq M$, }\\
	0, & \text{if $r\geq M+1$.}
	\end{cases}
\ees
Moreover, we choose $\chi^{M}$ in such a way that $\sup_{r\geq 0}|\frac{\De}{\De r}\chi^{M}(r)|<+\infty$. Next, we consider the Hamiltonian
\be\label{eq:opt_cond_1}
H^M(x,p)=\inf_{|u| \leq M} F(x,u)+b^{M}(x,u)\cdot p.
\ee
  For any $M$, we can invoke \cite[Ch. IV, Thm 4.3]{fleming2006controlled} to obtain existence of a unique classical solution $\varphi^{M,T,g} \in C^{1,2}_p([0,T)\times\bbRD)$ of 
\be\label{eq:HJB_M}
\begin{cases}
\partial_t \varphi_t +\frac{\sigma^2}{2}\Delta  \varphi_t -H^M(x,\nabla\varphi_t(x))=0,\\
 \varphi_T=g.
\end{cases}
\ee
Furthermore, an application of \cite[Ch IV, Lemma 8.1]{fleming2006controlled} gives the existence of a constant $K_1\in(0,+\infty)$ such that
\be\label{eq:opt_cond_3}
\sup_{M\in\mathbb{N}}\sup_{\substack{x\in \bbRD \\ 0\leq t\leq T }} | \nabla\varphi^{M,T,g}_t(x)| \leq K_1. 
\ee
For any $(x,p)$, let $w^M(x,p)$ be an optimizer in \eqref{eq:opt_cond_1}. Abbreviating $w^M(x,\nabla\varphi^{M,T,g}_t(x))$ with $w^M$ and imposing 
\bes
b^M(x,w^M)\cdot \nabla\varphi^{M,T,g}_t(x)+F(x,w^M)\leq b^M(x,(1-\varepsilon)w^M)\cdot \nabla\varphi^{M,T,g}_t(x)+F(x,(1-\varepsilon)w^M)
\ees and eventually letting $\varepsilon\downarrow 0$ we find
\bes
D_u F(x,w^M)\cdot w^M \leq  -(\nabla\varphi^{T,M,g}_t(x))^{\top} \cdot D_u b^M(x,w^M)\cdot w^M  
\ees
But then, using \eqref{eq:SP_Fconv_ass} we find that there exists $K_2\in(0,+\infty)$ with the property that
\be\label{eq:SP_opt_cond_8}
\sup_{M\in\N}\sup_{\substack{x\in \bbRD, |p|\leq K_1 \\ 0\leq t\leq T }} |w^M(x,p)| \leq K_2.
\ee
We now introduce the Hamiltonian
\bes
\bar H:\bbR^{d+d}\longrightarrow \bbR,\quad \bar{H}(x,p) = -\inf_{|u|\leq K_2} b(x,u) \cdot p+F(x,u).
\ees
As a result of \eqref{eq:SP_opt_cond_3} and \eqref{eq:SP_opt_cond_8} we have that for all $M,M'\geq K_2$ we find that
\be\label{eq:SP_opt_cond_9}
\begin{split}
H^M(x,\nabla\varphi^{M,T,g}_t(x))= \bar H(x,\nabla\varphi^{M,T,g}_t(x)),\\
H^M(x,\nabla\varphi^{M,T,g}_t(x))= H^{M'}(x,\nabla\varphi^{M,T,g}_t(x))
\end{split}
\ee
hold uniformly on $(t,x)\in [0,T)\times\bbRD$. 
But then, by uniqueness of solutions for \eqref{eq:HJB_M} we find that
\bes
\varphi^{M,T,g}\equiv\varphi^{M',T,g}:=\bar\varphi^{T,g}, \quad M,M'\geq K_2.
\ees
and that $\bar\varphi^{T,g}$ is a solution of \eqref{eq:HJB_M} for the Hamiltonian $\bar H$ in $C^{1,2}_p([0,T)\times\bbRD)$. As a consequence, for all $M\geq K_2$ we have  
\bes
    \bar H(x,\nabla\bar\varphi^{T,g}_t(x))=H^{M}(x,\nabla\bar\varphi^{T,g}_t(x))= \sup_{|u|\leq M} b(x,u)\cdot \nabla \bar\varphi^{T,g}_t(x)+F(x,u)
\ees
But then, letting $M\rightarrow+\infty$ in \eqref{eq:SP_opt_cond_9} we find 
\bes
\bar H(x,\nabla\bar\varphi^{T,g}_t(x)) = \sup_{u\in\bbRD} b(x,u)\cdot \nabla \bar\varphi^{T,g}_t(x)+F(x,u)=H(x,\nabla\bar\varphi^{T,g}_t(x)).
\ees
We have therefore shown that $\bar\varphi^{T,g}$ is a classical solution for equation \eqref{eq:SP_HJB} with the desired regularity properties. Let now $\tilde\varphi^{T,g}\in\{\varphi\in C^{1,2}_{p,\lip}([0,T)\times\bbRD): \sup_{s\in[0,T]}\|\varphi_s\|_{\lip}<+\infty\}$ be another classical solution of \eqref{eq:SP_HJB}, that may or may not coincide with $\bar\varphi^{T,g}$. Then, it follows directly from Assumption \ref{ass:SP_wellposed}-\ref{ass:SP} that the drift field 
\bes
[t,T]\times\bbRD\ni(s,x)\mapsto -D_p H(x,\nabla\tilde\varphi^{T,g}_s(x)) =b(x,w(x,\nabla\tilde\varphi^{T,g}_s(x)))
\ees
is locally Lipschitz and with linear growth in the space variable. But then, we know that for any $(t,x)$ there exist a unique strong solution for 
\bes
\begin{cases}
\De X_s= - D_pH(X_s, \nabla\tilde\varphi^{T,g}_s(X_s))\De s +\sigma\De B_s,\\
X_t=x.
\end{cases}
\ees
which proves item $(iii)$. Using the Lipschitzianity of $\tilde{\varphi}^{T,g}$ and our assumptions, we find that $u=w(X_s,\nabla\tilde\varphi^{T,g}_s(X_s))\in\cU_{[t,T]}.$ At this point, we can apply a verification result such as \cite[Ch. IV,  Thm. 3.1]{fleming2006controlled} from which the uniqueness statement in $(i)$ as well as well as the proofs of the statements at item $(ii)$ and $(iv)$ follow at once.
\end{proof}

\subsection{Proof of Lemma \ref{lem:coeff_bound}}
\begin{proof}
Imposing
\bes 
\ip{\partial_uF(x,w(x,p))+D_ub(x,w(x,p))\cdot p-\partial_uF(x,0)+D_ub(x,0)\cdot p}{w(x,p)} \geq \omega_{|p|}|w(x,p)|^2
\ees
and combining it with 
\be\label{eq:SP_coeff_bound_1}
\partial_uF(x,w(x,p))+ D_ub(x,w(x,p))\cdot p=0
\ee
we obtain from \eqref{eq:SP_drift_ass_2} that
\bes
|w(x,p)|\leq M_u\frac{(1+|p|)}{\omega_{|p|}},
\ees
from which \eqref{eq:control_bound_1} follows thanks to the gradient estimate \eqref{eq:SP_grad_est_7}. The relation \eqref{eq:drift_lipschitz_1} is a direct consequence of \eqref{eq:control_bound_1} and \eqref{eq:SP_drift_ass_2}. To prove \eqref{eq:Hess_Ham_bound}, we first differentiate \eqref{eq:SP_coeff_bound_1} w.r.t. $p$ to find
\bes
\Big(\partial^2_{uu}F(x,w(x,p))+ D_{uu} b(x,w(x,p))\cdot p\Big)\cdot D_p w(x,p)= - D_ub(x,w(x,p))\cdot 
\ees
At this point,  \eqref{eq:SP_Fconv_ass} gives
\bes
|D_pw(x,p)| \leq \omega^{-1}_{|p|}M_u \qquad \forall x,p\in\bbRD,
\ees
which is \eqref{eq:coeff_bound_13}. Next, we observe that
\bes
-D_{pp}H(x,p)= D_ub(x,w(x,p))\cdot D_pw(x,p) \leq M^2_u \omega^{-1}_{|p|},
\ees
that is precisely \eqref{eq:Hess_Ham_bound}.
Let us move to the proof of \eqref{eq:coeff_bound_4}. Observing that
\bes\label{eq:SP_coeff_bound_6}
-D_xH(x,p) = D_xb(x,w(x,p))\cdot p + D_xF(x,w(x,p)),
\ees
we obtain from the current hypothesis that
\bes
|D_xH(x,p)|\leq M_x|p|+M_x^F.
\ees
At this point, \eqref{eq:coeff_bound_4} follows from the gradient bound \eqref{eq:SP_grad_est_7}. Let's proceed to the proof of \eqref{eq:coeff_bound_2} and \eqref{eq:coeff_bound_3}. To do so, we first differentiate  \eqref{eq:SP_coeff_bound_1} w.r.t. to the position variables to find
\bes
\begin{split}
\Big(\partial^2_{uu}F(x,w(x,p))+ D_{uu} b(x,w(x,p))\cdot p\Big)\cdot D_x w(x,p)= -D_{xu}b(x,w(x,p))\cdot p-D_{xu}F(x,w(x,p))
\end{split}
\ees
from which we obtain \eqref{eq:coeff_bound_11} and
\be\label{eq:coeff_bound_9}
|D_xw(x,\nabla\varphi^{T,g}_s(x))| \leq  \frac{M_{xu}(1+M^{\varphi,g}_x)}{\omega_{M^{\varphi,g}_x}} \quad \forall x\in\bbRD,0\leq s\leq T.
\ee
At this point \eqref{eq:coeff_bound_2} follows observing that 
\bes
\begin{split}
|D_{xp}H|(x,\nabla\varphi^{T,g}_s(x)) &\leq |D_xb(x,w(x,\nabla\varphi^{T,g}_s(x)))|\\
&+|D_ub(x,w(x,\nabla\varphi^{T,g}_s(x)))||D_xw(x,w(x,\nabla\varphi^{T,g}_s(x)))|\\
&\leq M_x+M_u |D_xw(x,\nabla\varphi^{T,g}_s(x))| \\
&\stackrel{\eqref{eq:coeff_bound_9}}{\leq} M_x+M_u \frac{M_{xu}(1+M^{\varphi,g}_x)}{\omega_{M^{\varphi,g}_x}} .
\end{split}
\ees
The bound \eqref{eq:coeff_bound_3} is obtained in a similar way. Indeed, starting from the identity
\bes
\begin{split}
D_{xx}H(x,p) &= D_{xx}b(x,w(x,p))\cdot p+ D_xw(x,p)\cdot(D_{xu}b(x,w(x,p))\cdot p)\\
                &+D_{xx}F(x,w(x,p))+ D_xw(x,p)\cdot D_{xu}F(x,w(x,p))
\end{split}
\ees
we obtain
\bes
\begin{split}
|D_{xx}H|(x,\nabla\varphi^{T,g}_s(x)) &\leq M_{xx}|\nabla\varphi^{T,g}_s(x)|+M_{xu}|\nabla\varphi^{T,g}_s(x)||D_xw(x,\nabla\varphi^{T,g}_s(x))|\\
&+M_{xx}+ M_{xu}|D_xw(x,\nabla\varphi^{T,g}_s(x))|\\
&\stackrel{\eqref{eq:SP_grad_est_7},\eqref{eq:coeff_bound_9}}{\leq} M_{xx}(1+M_x^{\varphi,g})+\frac{M^2_{xu}(1+M^{\varphi,g}_x)^2}{\omega_{M^{\varphi,g}_x}}.
\end{split}
\ees
\end{proof}

\subsection{Proof of Proposition \ref{prop:aux_hess_bd}}
\begin{proof}
The bounds involving $H,D_xH$ are straighforward consequence of \eqref{eq:drift_lipschitz_1},\eqref{eq:coeff_bound_4}. A stronger version of the gradient bound on $\varphi^{T,g}$ has already been proven at Lemma \ref{lemma:SP_grad_est}. This estimate implies the desired linear bound on $\varphi^{T,g}_t(x)$. In order to establish the linear growth of $\partial_t\varphi^{T,g}_t(x)$ we observe that because of Proposition \ref{prop:SP_opt_cond} and \eqref{eq:control_bound_1} $\varphi^{T,g}_t(x)$ is the optimal value of the problem obtained adding to \eqref{eq:SP_prob} the additional constraint that admissible controls must satisfy 
\bes
\bbP-\text{a.s.} \quad |u_s |\leq \frac{M_u(1+M^{\varphi,g}_x)}{\omega_{M^{\varphi,g}_x}}\quad\forall s\in[t,T].
\ees
Using this equivalent formulation of \eqref{eq:SP_prob}, all hypothesis needed to apply \cite[Thm 3, Sec 4.4]{krylov2008controlled} are satisfied and the linear growth of $\partial_t\varphi^{T,g}_t(x)$ follows from this result. It remains to show the bound on the Hessian. To this aim, we observe that the standing assumptions enables to apply the semiconcavity estimate \cite[Sec 4.2, Thm 3]{krylov2008controlled} giving the existence of some constant $C'_T\in(0,+\infty)$ such that the one-sided bound
\be\label{eq:aux_hess_bd_2}
\forall\, (t,x) \in[0,T]\times\bbRD \quad \nabla^2 \varphi^{T,g}_t(x) \leq C'_T \mathrm{I}
\ee
holds, where the above inequality is to be an understood in the sense of quadratic forms.
To conclude, we observe that using the bounds on $\partial^{T,g}_t\varphi^{T,g},\nabla\varphi^{T,g},$ and the HJB equation \eqref{eq:SP_HJB}, we find that, enlarging the value of $C'_T$ if necessary, 
\bes
\forall\, (t,x) \in[0,T]\times\bbRD \quad |\Delta\varphi^{T,g}_t(x)|\leq C'_T
\ees
But then, combining this last bound with \eqref{eq:aux_hess_bd_2} we obtain the linear growth of $|\nabla^2\varphi^{T,g}_t|$.
\end{proof}

\subsection{Proof of Propositon \ref{prop:policy_opt}}
\begin{proof}
Consider an approximating sequence $g^M\subseteq C^3_{\lip}(\bbRD)$ such that 
\bes 
\lim_{M\rightarrow +\infty} \sup_{x\in\bbRD}|g^M-g|(x)=0,\quad \lim_{M\rightarrow+\infty} \|g^M -g\|_{\mathrm{Lip}}=0.
\ees
The existence of such a sequence is granted by \cite[Prop. A item (b)]{czarnecki2006approximation}. From these properties and Lemma \ref{lemma:SP_stab_fin_cond} it follows that 
\be\label{eq:policy_opt_3}
\begin{split}
\lim_{M\rightarrow+\infty}\sup_{\substack{t\in[0,T]\\x\in\bbRD}}|\varphi^{T,g^M}_t -\varphi^{T,g}_t |(x) =0,\\ \lim_{M\rightarrow+\infty}\sup_{\substack{t\in[0,T]\\x\in\bbRD}}|\nabla\varphi^{T,g^M}_t -\nabla\varphi^{T,g}_t |(x) =0
\end{split}
\ee
In particular, the second identity in the above implies that
\be\label{eq:policy_opt_1}
\lim_{M\rightarrow+\infty}\sup_{\substack{t\in[0,T]\\x\in\bbRD}}|\beta^{T,g^M}_t -\beta^{T,g}_t |(x) =0
\ee
Fix now $\varepsilon>0$. Since $\beta^{T,g}$ is uniformly Lipschitz in space and continuous in time (see \eqref{eq:reg_eff_3}) on $[0,T-\varepsilon]\times\bbRD$, we deduce from \eqref{eq:policy_opt_1}  and a standard application of Gr\"onwall's Lemma that 

\be\label{eq:policy_opt_2}
\bbP-a.s. \quad \limsup_{M\rightarrow+\infty}  \sup_{s\in[0,T-\varepsilon]} |\bmX^{0,x,T,g^M}_s-\bmX^{0,x,T,g}_s|=0.
\ee
Moreover, combining the standing assumption with \eqref{eq:control_bound_1} and \eqref{eq:policy_opt_2} we obtain that there exist $C\in(0,+\infty)$ such that
\be\label{eq:policy_opt_5}
\bbP-\text{a.s.} \quad \sup_{M\in\N,s\in[0,T]} F(\bmX^{0,x,T,g^M}_s,w(\bmX^{0,x,T,g^M}_s,\nabla\varphi^{T,g^M}_s(\bmX^{0,x,T,g^M}_s))) \leq C\big(1+ \sup_{s\in[0,T]}|\bmX^{0,x,T,g}_s|\big)
\ee
and the right hand side is an integrable random variable because $\bar\kappa_{\beta^{T,g}}\in K$. But then, by dominated convergence we find that
\bes
\begin{split}
\lim_{M\rightarrow+\infty} \bbE\Big[\int_0^{T-\varepsilon}F(\bmX^{0,x,T,g^M}_s,w(\bmX^{0,x,T,g^M}_s,\nabla\varphi^{T,g^M}_s(\bmX^{0,x,T,g^M}_s)))\De s
\\
	-\int_0^{T-\varepsilon}F(\bmX^{0,x,T,g}_s,w(\bmX^{0,x,T,g}_s,\nabla\varphi^{T,g}_s(\bmX^{0,x,T,g}_s)))\De s\Big]=0.
\end{split}
\ees
Next, observe that using \eqref{eq:sticky_coup_1} with $s=T$ and the lipschitzianity of $g$ gives
\bes
\lim_{M\rightarrow+\infty} |\bbE[g^M(\bmX^{0,x,T,g^M}_T)]-\bbE[g(\bmX^{0,x,T,g}_T)]|=0.
\ees
But then,
\bes
\begin{split}
&\lim\sup_{M\rightarrow+\infty} |J^{T,g^M}_{t,x}(w(\bmX^{0,x,T,g^M}_s,\nabla\varphi^{T,g^M}_s(\bmX^{0,x,T,g^M}_s))_{s\in[0,T]})\\
&-J^{T,g}_{t,x}(w(\bmX^{0,x,T,g}_s,\nabla\varphi^{T,g}_s(\bmX^{0,x,T,g}_s))_{s\in[0,T]})|\\
&\leq \lim\sup_{M\rightarrow+\infty} \bbE\Big[\int_{T-\varepsilon}^T|F(\bmX^{0,x,T,g^M}_s,w(\bmX^{0,x,T,g^M}_s,\nabla\varphi^{T,g^M}_s(\bmX^{0,x,T,g^M}_s)))\\
&-F(\bmX^{0,x,T,g}_s,w(\bmX^{0,x,T,g}_s,\nabla\varphi^{T,g}_s(\bmX^{0,x,T,g}_s)))|\De s\Big]\\
&\stackrel{\eqref{eq:policy_opt_5}}{\leq} 2\varepsilon C\, \Big(1+\bbE\Big[ \sup_{s\in[0,T]}|\bmX^{0,x,T,g}_s|\Big]\Big)
\end{split}
\ees
Since $\bbE[ \sup_{s\in[0,T]}|\bmX^{0,x,T,g}_s|]<+\infty$ and $\varepsilon>0$ can be chosen arbitrarily, we obtain from Proposition \ref{item_4:SP_opt_cond} that
\bes
\lim\sup_{M\rightarrow+\infty} |\varphi^{T,g^M}_0(x)
-J^{T,g}_{t,x}(w(\bmX^{0,x,T,g}_s,\nabla\varphi^{T,g}_s(\bmX^{0,x,T,g}_s))_{s\in[0,T]})|=0
\ees
The desired conclusion follows from the first identity in \eqref{eq:policy_opt_3}.
\end{proof}
\subsection{On the continuity of $\bar\Phi_T$}
\begin{lemma}\label{lem:rest_cont}
Let $T>0$ and $\bar\Phi_T:B_M\longrightarrow C_0(\bbRD)$ be given by \eqref{eq:erg_HJB_4}. Then $\bar\Phi_T$ is continuous, i.e. if $(g^k)_{k\geq 1}\subseteq B_M$ converges uniformly on compact sets to $g$, then $\bar\Phi_T(g^k)$ converges uniformly on compact sets to $\bar\Phi_T(g)$.
\end{lemma}
\begin{proof}
Fix $T>0$ and let $(g^k)_{k\geq 1}\subseteq B_M$ converge uniformly on compact sets to $g$. Fix $\varepsilon>0$ and let $g^{k,\varepsilon}=g^{k}\ast \gamma_{\varepsilon}$, $g^{\varepsilon}=g\ast\gamma_{\varepsilon}$ with $\gamma_{\varepsilon}$ as in \eqref{eq:SP_Hess_est_11}. Then $(g^{k,\varepsilon})_{k\geq 1} \in B_M\cap C^3_{\lip}(\bbRD)$ and 
\be\label{eq:rest_cont_1}
\forall \varepsilon>0,\quad \sup_{x\in\bbRD,k\geq 1, t\in[0,T]}|\varphi^{T,g^k}_t-\varphi^{T,g^{k,\varepsilon}}_t|\leq \varepsilon^{1/2}, \quad \lim_{k\rightarrow+\infty}\sup_{x\in U}|g^{k,\varepsilon}-g^{\varepsilon}|=0, \quad ,
\ee
where $U$ denotes an arbitrarily chosen compact subset of $\bbRD$. Since $g^,g^{\varepsilon}\in C^3_{\lip}(\bbRD)$, the processes $(\bmX^{0,x,T,g^{k,\varepsilon}}_s)_{s\in[0,T]}$ and  $(\bmX^{0,x,T,g^{\varepsilon}}_s)_{s\in[0,T]}$ are well defined for all $x$. Moreover, since $\|g^{k,\varepsilon}\|_{f}\leq M$ uniformly on $k,\varepsilon$ and $x\mapsto b(x,0)\in K$, we obtain from \eqref{eq:drift_lipschitz_1} and some Gronwall's Lemma that for any compact set $U$
\bes
\sup_{k\geq 1,\varepsilon>0}\,\sup_{x\in U}\bbE[|\bmX^{0,x,T,g^{k,\varepsilon}}_s|^2+|\bmX^{0,x,T,g^{\varepsilon}}_s|^2] \leq C
\ees
for some $C\in(0,+\infty)$. We have for any $R,\varepsilon>0$ 
\bes
\begin{split}
    \sup_{x\in U}|\varphi^{T,g^{k}}_0(x)-\varphi^{T,g}_0(x) |&\leq 2\varepsilon^{1/2}+  \sup_{x\in U}|\varphi^{T,g^{k,\varepsilon}}_0(x)-\varphi^{T,g^{\varepsilon}}_0(x)|\\
    &\leq 2\varepsilon^{1/2}+  \sup_{x\in U}\bbE\Big[|g^{k,\varepsilon}-g^{\varepsilon}|(\bmX^{0,x,T,g^{k,\varepsilon}}_T) \Big] +\bbE\Big[|g^{k,\varepsilon}-g^{\varepsilon}|(\bmX^{0,x,T,g^{\varepsilon}}_T) \Big]\\
     & \leq 2\varepsilon^{1/2}+ 2\sup_{|y|\leq R}|g^{k,\varepsilon}-g^{\varepsilon}|(y)+ 2M\sup_{x\in U}\bbE\Big[|\bmX^{0,x,T,g^{k,\varepsilon}}_T| \mathbf{1}_{\{|\bmX^{0,x,T,g^{k,\varepsilon}}_T|\geq R\}} \Big]\\
     &+\sup_{x\in U}2M\bbE\Big[|\bmX^{0,x,T,g^{\varepsilon}}_T| \mathbf{1}_{\{|\bmX^{0,x,T,g^{\varepsilon}}_T|\geq R\}} \Big]\\
     &\leq 2\varepsilon^{1/2}+2\sup_{|y|\leq R}|g^{k,\varepsilon}-g^{\varepsilon}|(y)+2M C\sup_{x\in U}\bbP[|\bmX^{0,x,T,g^{k,\varepsilon}}_T|\geq R]\\
     &+2MC  \sup_{x\in U}\bbP[|\bmX^{0,x,T,g^{\varepsilon}}_T|\geq R])\\
     &+  2\varepsilon^{1/2}+2\sup_{|y|\leq R}|g^{k,\varepsilon}-g^{\varepsilon}|(y)+\frac{4M C^2}{R}
     \end{split}
\ees
But then, using \eqref{eq:rest_cont_1} we find
\bes
\limsup_{k\rightarrow+\infty}\sup_{x\in U}|\varphi^{T,g^{k}}_0(x)-\varphi^{T,g}_0(x)| \leq 2\varepsilon^{1/2}+\frac{4M C^2}{R}
\ees
Since $\varepsilon,R$ can be chosen arbitrarily, it follows that $\varphi^{T,g^{k}}_0(x)$ converges uniformly on compact sets to $\varphi^{T,g}_0(x)$, which gives the desired result.
\end{proof}


\bibliographystyle{imsart-nameyear.bst} 
\bibliography{ArXiv.bbl}       


\end{document}